\documentclass{amsart}
\usepackage{amsmath, amsthm, amscd, amsfonts}

\usepackage[margin=3cm]{geometry}

\usepackage[mathscr]{eucal}
\usepackage[bookmarks=false]{hyperref}
\usepackage{url}
\usepackage{enumerate}
\usepackage{bm}
\usepackage{sansmath}

\usepackage{ifpdf}
\ifpdf
	\usepackage[all,pdf,cmtip]{xy} 
\else
	\input xy
	\xyoption{all}
	\xyoption{2cell}
	\xyoption{v2}
	\xyoption{cmtip}
\fi

\CompileMatrices

\allowdisplaybreaks 

\newcommand{\diagram}{}
\def\diagram#1#{\diagramaux{#1}}
\newcommand\diagramaux[2]{\begin{gathered}\xymatrix#1{#2}\end{gathered}}

\newcommand{\J}{\mathsf{J}}
\newcommand{\into}{\hookrightarrow}    
\newcommand{\Cat}{\bm{\mathrm{Cat}}}
\newcommand{\Gpd}{\bm{\mathrm{Gpd}}}
\newcommand{\Bicat}{\bm{\mathrm{Bicat}}}
\newcommand{\equals}{\quad=\quad}   
\newcommand{\ff}{\mathsf{ff}}

\DeclareMathOperator{\id}{id}
\DeclareMathOperator{\pr}{pr}

\newtheorem{theorem}{Theorem}[section]
\newtheorem{lemma}[theorem]{Lemma}

\newtheorem{cor}[theorem]{Corollary}
\newtheorem{prop}[theorem]{Proposition}
\theoremstyle{definition}
\newtheorem{definition}[theorem]{Definition}
\theoremstyle{definition}
\newtheorem{remark}[theorem]{Remark}
\theoremstyle{definition}

\theoremstyle{definition}
\newtheorem{example}[theorem]{Example}

\begin{document}

\title[Formal anafunctors]{The elementary construction of formal anafunctors}

\author[D.M.\ Roberts]{David Michael Roberts}

\address{School of Mathematical Sciences, The University of Adelaide, Adelaide, SA 5005, Australia}

\email{david.roberts@adelaide.edu.au}


\thanks{This document is under a Creative Commons Attribution license: \href{https://creativecommons.org/licenses/by/4.0/}{\texttt{creativecommons.org/licenses/by/4.0/}}}

\begin{abstract}
This article gives an elementary and formal 2-categorical construction of a bicategory of right fractions analogous to anafunctors, starting from a 2-category equipped with a family of covering maps that are fully faithful and co-fully faithful.
\end{abstract}

\maketitle

\section{Introduction}

Anafunctors were introduced by Makkai \cite{Makkai} as new 1-arrows in the 2-category $\Cat$ to talk about category theory in the absence of the axiom of choice.
The aim was to make functorial those constructions that are only defined by some universal property, rather than by some specified operation. One also recovers the characterisation of equivalences of categories as essentially surjective, fully faithful 1-arrows.
The construction by Bartels \cite{Bartels} of the analogous bicategory $\Cat_{ana}(S,J)$, whose 1-arrows are anafunctors, starting from the 2-category $\Cat(S)$ of internal categories was extended in \cite{Roberts_12} to variable full sub-2-categories  $\Cat'(S) \into \Cat(S)$. 
The canonical inclusion 2-functor $\Cat'(S) \into \Cat'_{ana}(S,J)$ was there shown to be a  2-categorical localisation in the sense of Pronk \cite{Pronk_96} at the fully faithful functors which are locally weakly split in the given pretopology $J$ on $S$.

In these notes I show that given a 2-category $K$ equipped with a (strict) singleton pretopology $\J$ whose elements are fully faithful and co-fully faithful arrows, one can construct an analogue  $K_\J$ of the bicategory $\Cat'_{ana}(S,J)$.
The 1-arrows of $K_\J$ are formal 2-categorical versions of anafunctors, here dubbed \emph{$\J$-fractions}.
The construction of $K_\J$ is elementary in the sense of only needing the first-order theory of 2-categories, and the construction is Choice-free. 
The original 2-category $K$ is a wide and locally full sub-bicategory of $K_\J$ and the inclusion 2-functor $A_\J\colon K \into K_\J$ is a bicategorical localisation; this result uses Pronk's comparison theorem from \cite{Pronk_96}, but it should be possible to prove directly using the construction given here.

The following quote from \cite{Simpson_06} should be kept in mind when reading the elementary calculations in these notes, as no such details have fully appeared in the literature, let alone at the level of generality here:

\begin{quote}

	Nonetheless, it is interesting to note the prevalence of formulations leaving ``to the reader'' parts of the proofs of details of the localization constructions. \ldots\ 
	Another interesting reference is Pronk's paper on localization of 2-categories [21]\footnote{\cite{Pronk_96} in the References below.}, pointed out to me by I.\ Moerdijk. 
	This paper constructs the localization of a 2-category by a subset of 1-morphisms satisfying a generalization of the right fraction condition.
	\ldots\ the full set of details for the coherence relations on the level of 2-cells is still too much, so 
	the paper ends with:

	\begin{itemize}
		\item[[21], p. 302:] ``It is left to the reader to verify that the above defined isomorphisms $a$, $l$ and $r$ are natural in their arguments and satisfy the identity coherence axioms.'' 
	\end{itemize}

\end{quote}

One pleasant feature of the current approach, at least for the author, is that one could take the opposite 2-category everywhere in the current notes and everything will still work fine, only exchanging pullbacks for pushouts everywhere. In this way, one could also localise suitable 2-categories using \emph{cospans}, rather than spans, for instance 2-categories whose objects are more algebraic in nature, rather than geometric, like Hopf algebroids.
It is not clear that for alternative presentations of the localisation, say one using of left-principal bibundles, such an approach would still work, or what should play the r\^ole of the 1-arrows when one is working with a general 2-category and not a 2-category of structured or internal groupoids.

\begin{remark}
This article has had a long and tortuous history. 
It started out as the second half of what was published as \cite{Roberts_14a}, in an attempt to give an elementary and completely self-contained proof ommitting no details of the results of \cite{Pronk_96} dealing with 2-categories of stacks as bicategorical localisations.
In this the author was partially influenced by the late Vladimir Voevodsky's insistence on details and constructions in what might otherwise seem merely bureaucratic proofs.
The appearance of \cite{Pronk-Warren_14} gave a much more satisfactory (and conceptual) proof of these results, after which the author lost hope that there was any merit in the current naive approach.
However, work in preparation using cospans to localise 2-categories of topological groupoids required the machinery of the present paper, in its dual incarnation, so it is hoped there is some merit in the elementary approach in it.
\end{remark}

\section{Preliminaries}

We refer to \cite[Chapter 2]{Johnson-Yau} for background on bicategories.

\begin{definition}
	\label{def:representably_P}

	If $\mathcal{P}$ is a property of functors, then we say that a 1-arrow
	$f\colon x\to y$ in a 2-category $K$ is \emph{representably
	$\mathcal{P}$} (or just $\mathcal{P}$) if for all objects $z$ of $K$ we have that $f_*\colon
	K(z,x)\to K(z,y)$ has property $\mathcal{P}$.

\end{definition}

The most important case for the present paper is the property `fully faithful', and we will abbreviate `representably fully faithful' to ff, and will denote by $\ff$ the class of ff 1-arrows in a 2-category. 
Note however that Definition~\ref{def:representably_P} can be rewritten as a first-order property of a 1-arrow in a 2-category (and even in a bicategory).

\begin{definition}
	\label{def:ff_arrow}

	A 1-arrow $f\colon x\to y$ in a bicategory is ff if for all $g, h\colon w\to x$ and $\widetilde{a}\colon f\circ g \Rightarrow f\circ h$ there is a unique $a\colon g\Rightarrow h$ such that $\widetilde{a} = 1_f \circ a$. 

\end{definition}

\begin{example}
Any equivalence in a 2-category is ff; it is a neat exercise to directly construct the required 2-arrow $a$.
It is less easy to do the analogous construction for an equivalence in a bicategory, showing that such a 1-arrow is ff in an elementary fashion, but still possible.\footnote{In fact the computation shows that a 1-arrow with a representably faithful pseudo-retract is ff, but we don't need this level of generality.}
\end{example}


\begin{lemma}
\label{lemma:restricted_2of3_for_ff}
	If $f\colon y\to z \in \ff$ and $g\colon x \to y$ is any other arrow then $g\in \ff$ if and only if $f\circ g \in \ff$. 
	If $h\colon y\to z$ is another arrow that is isomorphic to $f$ in $K(y,z)$ then $h\in \ff$. 
	Moreover, if $f$ is isomorphic to $f'\colon y'\to z'$ (in the arrow 2-category) then $f'\in \ff$.
\end{lemma}

The following lemma will be a major workhorse in the construction below.

\begin{lemma}
	\label{lemma:unique_lifts}

	Let $p\colon u\to a$ be an ff 1-arrow in a 2-category $K$. 
	Then for any 1-arrow $f\colon b\to a$ and any two lifts $k,l\colon b \to u$ of $f$ through $p$, there is a unique 2-arrow $k\Rightarrow l$ lifting the identity 2-arrow on $f$ through $p$.

\end{lemma}

The proof of this lemma follows almost immediately from the definition of ff arrows.

\begin{example}
	\label{eg:unique_lifts_0}

	Given an ff 1-arrow $p\colon u\to x$ in $K$ such that the strict pullback $u\times_x u$ exists, then there is a unique 2-arrow $\ell_p\colon \pr_1 \Rightarrow \pr_2\colon u\times_x u \to u$ lifting $\id\colon p\Rightarrow p$.
\end{example}

In this paper, all pullbacks are likewise \textbf{strict}.

\begin{example}
	\label{eg:unique_lifts_1}

	Consider a commutative triangle
	\[
		\xymatrix{
			u \ar[rr]^\phi \ar[dr]_p && v \ar[dl]^q\\
			& x
		}
	\]
	with $q\colon v\to x$ ff. 
	Then assuming the relevant strict pullbacks exist, there is an equality
	\[
		\xymatrix{
			u\times_x u \ar@/^1pc/[r]_{\ }="s1"^{\pr_1} \ar@/_1pc/[r]^{\ }="t1"_{\pr_2}
			& u \ar[r]^\phi & v \ar@{=>}"s1";"t1"^{\ell_p}
		}
		\equals
		\xymatrix{
			u\times_x u \ar[r]^-{\phi\times\phi} & v\times_x v
			\ar@/^1pc/[r]_{\ }="s2"^{\pr_1} \ar@/_1pc/[r]^{\ }="t2"_{\pr_2} & v
			\ar@{=>}"s2";"t2"^{\ell_q}
		}
	\]
	between the pasted 2-cells, as the source and target 1-arrows all lift $u\times_x u \to x$ through $q$.

\end{example}

\begin{example}
	\label{eg:unique_lifts_2}

	A more complicated example is the equality of pasted 2-cells in
	\[
		\vcenter{\xymatrix{
		&& u\times_x u \ar@/_1pc/[dr]^{\ }="t1"_(0.35){\pr_2} \ar@/^1pc/[dr]_{\ }="s1"^{\pr_1}\\
		u\times_x u \ar[r]^-{\Delta\times\id_u} & u\times_x u \times_x u
		\ar@/^1pc/[ur]^{\pr_{12}} \ar@/_1pc/[dr]_{\pr_{23}}&& u\\
		&& u\times_x x \ar@/^1pc/[ur]_{\ }="s2"^(0.35){\pr_1} \ar@/_1pc/[ur]^{\ }="t2"_{\pr_2}
		\ar@{=>}"s1";"t1"^{\ell_p}
		\ar@{=>}"s2";"t2"^{\ell_p}
		}}
		\equals
		\vcenter{\xymatrix{
		u\times_x u \ar@/^1pc/[r]_{\ }="s3"^{\pr_1} \ar@/_1pc/[r]^{\ }="t3"_{\pr_2} & u
		 \ar@{=>}"s3";"t3"^{\ell_p}
		}}
	\]
\end{example}

The structure of a site on a 2-category is not a common notion so we need to specify what we mean. 
There are at least two different ways to describe this in the 1-categorical case, namely using sieves and using pretopologies, and it is not clear a priori that they generalise to the same thing for 2-categories. 
Our definition will be as follows, as this paper only deals with unary sites.

\begin{definition}
	\label{def:strict_pretopology}

	A \emph{singleton strict pretopology} on a 2-category $K$ is a class $\J$ of 1-arrows which contains all identity arrows, is closed under composition and the strict pullback an element of $\J$ exists and is again in $\J$. 
	We will assume that \emph{specified} strict pullbacks are given---rather than merely assuming they exist---and that the pullback of an identity 1-arrow is again an identity 1-arrow.

\end{definition}

Since this is the same thing as a singleton pretopology on the 1-category underlying the 2-category, we refrain from placing the prefix `2-' in the name.
If one merely asks for existence of pullbacks, then one may use a global axiom of choice to make the pullback of a cover an operation.%

\begin{example}

	Let $K$ be a 2-category which admits specified strict pullbacks. 
	Then $\ff$ is a singleton strict pretopology.

\end{example}

This is in some sense a degenerate example. 
The following is more of interest.

\begin{example}\label{ex:pretopology_on_cat_s}

	Let $S$ be a finitely complete category with specified limits and $\J_0$ a singleton pretopology on $S$. 
	Then we have the 2-categories $\Cat(S)$ and $\Gpd(S)$ of internal categories and groupoids. Let $\J$ denote the class of internal functors in either of those 2-categories whose object component is an arrow in $\J_0$. 
	Then $\J$ is a singleton strict pretopology on both $\Cat(S)$ and $\Gpd(S)$.

\end{example}

In addition, we need to consider a 2-categorical version somewhat analogous to subcanonicity, and here we cannot avoid involving the 2-arrows. 
This makes the notion essentially 2-categorical, and not just a structure on the underlying 1-category as is the case for Definition~\ref{def:strict_pretopology}.\footnote{For the sense in which this deserves to be considered a type of subcanonicity, cf \cite[Lemma 52]{GarnerBourke}, which states that an ff regular epimorphism is also co-ff. }

\begin{definition}\label{def:subcanonical pretop}

	Given a singleton strict pretopology $\J$, we will call it \emph{bi-ff} if every $j\colon u\to x$ in $\J$ is ff and also co-fully faithful (co-ff): for all $g, h\colon x\to y$ and $\widetilde{a}\colon g\circ j \Rightarrow h\circ j$ there is a unique $a\colon g\Rightarrow h$ such that $\widetilde{a} = 1_j \circ a$. 

\end{definition}

In this paper we do not need to descend 1-arrows down a cover in the pretopology (which is a consequence of representable presheaves being sheaves), but only 2-arrows, so the weaker notion of ff + co-ff is sufficient.
We will not here dwell on how this relates to 2-dimensional sheaf theory \`a la Street \cite{Street_82}.

\begin{example}
	Continuing Example~\ref{ex:pretopology_on_cat_s}, if we additionally assume the pretopology $\J_0$ consists of regular epimorphisms (hence is a subcanonical pretopology on $S$), then $\J$ is a bi-ff singleton strict pretopology on $\Cat(S)$ and on $\Gpd(S)$.
	Indeed, the arrows in this pretopology are also regular epimorphisms, though we do not need this here.
\end{example}

This example partly recovers the examples that were used in \cite[\S 8]{Roberts_12}; variants on this definition will give all examples from \emph{loc.\ cit}.

The main object of study of this paper are 2-categories $K$ with a choice of bi-ff singleton strict pretopology $\J$. 
We shall just refer to these as \textbf{2-sites} for brevity, though properly speaking it is a very special case of this notion.

As a consequence of our definition of 2-site, we don't just get descent of 2-arrows along covers, but along maps between covers.

\begin{lemma}\label{lemma:stronger_co-ff}
Given a 2-site $(K,\J)$, and a diagram
\[
	\xymatrix{
		v \ar[dr]_k \ar[rr]^r && w \ar[dl]^j \\
		& x
	}
\]
where $k,j\in \J$, then $r$ is also co-ff.
\end{lemma}
The following proof, simplifying the author's, is due to the anonymous referee.
\begin{proof}
Recall that $r$ being co-ff means that given a 2-arrow $\overline{a}\colon f\circ r \Rightarrow g\circ r$, there is a unique 2-arrow $a\colon f\Rightarrow g$ such that $a\circ \id_r = \overline{a}$. 
From our definition of 2-site, $j$ and $k$ are both ff and co-ff. 
This also implies that the two projection maps $\pr_1,\pr_2$ in the next diagram are in $\J$, hence are both co-ff.
Since $j\circ \pr_2 = j\circ r \circ \pr_2$ and $j$ is ff, there is a unique lift of $\id_{j\circ\pr_2}$ to an invertible 2-arrow $r\circ\pr_1 \stackrel{\sim}{\Rightarrow} \pr_2$:
\[
	\diagram{
	&  v \ar[dr]^r_(0.2){\ }="s" \\
	v\times_x w \ar[ur]^{\pr_1} \ar[rr]_{\pr_2}^(0.6){\ }="t" \ar@/_1pc/[drr]_{j\circ\pr_2} && w \ar[d]^j \\
	&&x
	\ar@{=>}"s";"t"_\simeq
	}
\]
But $\pr_2$ is co-ff, implying $r\circ \pr_1$ is co-ff, and since $\pr_1$ is co-ff, then so is $r$ (applying two of the cases of Lemma~\ref{lemma:restricted_2of3_for_ff} in $K^{op}$).
\end{proof}

\section{The bicategory of \texorpdfstring{$\J$}{J}-fractions}

We are aiming to localise a 2-category, and in time-honoured tradition we shall call the arrows in the localised 2-category fractions. 
Fractions are defined relative to a strict pretopology.

\begin{definition}

	Let $(K,\J)$ be a 2-site.  
	A \emph{$\J$-fraction} is a span $x\xleftarrow{j} u \xrightarrow{f} y$ in $K$ where $j\in \J$, to be denoted $(j,f)$.

\end{definition}

For example, given any 1-arrow $f\colon x\to y$ in $K$, we have the fraction $(\id_x,f)$. 
In particular, we have for any object $a$ the \emph{identity fraction}, which is $(\id_x,\id_x)$

\begin{definition}

	Let $x \xleftarrow{j} u \xrightarrow{f} y$ and $x \xleftarrow{k} v \xrightarrow{g} y$ be a pair of $\J$-fractions in $K$. A \emph{map of $\J$-fractions} $(j,f) \Rightarrow (k,g)$ consists of a 2-arrow
	\[
		\xymatrix{
			& u \ar[dr]^f_(.3){\ }="s" \\
			 u\times_x v \ar[ur]^{\pr_1}\ar[dr]_{\pr_2} & & y\\
			& v  \ar[ur]_g^(.3){\ }="t"
			\ar@{=>}"s";"t"_a
		}
	\]
Sometimes we will also write the 1-arrow $x \leftarrow u\times_x v$ in such a diagram for emphasis, so that the 0-source object is clear.\footnote{Note that, as presented, this 2-arrow alone does not allow us to reconstruct its (1-)source and (1-)target; we take the source and target as implicitly part of the data (cf the definition of arrow in the category of ZFC-sets).}
\end{definition}

There are certain maps of fractions which are easier to describe and to compose, and the coherence maps of the bicategory we are going to define all turn out to be examples, so we shall spend some time detailing these.

\begin{definition}\label{def:renaming_map}

	A \emph{renaming map}\footnote{The `renaming transformations' of \cite[\S1]{Makkai} in the case when $K=\Cat$ are a special case of the notion here. Makkai requires that $r$ is invertible and $a_r$ is the identity 2-arrow.} $r$ from the fraction $(j,f)$ to the fraction $(k,g)$ is a map of spans in a 2-category of the form:
	\[
		\xymatrix{
			& \ar[dl]_j u \ar[dr]^f_{\ }="s" \ar[dd]_r\\
			x && y\\
			& \ar[ul]^k v \ar[ur]_g^{\ }="t"
			\ar@{=>}"s";"t"^{a_r}
		}
	\]
	We can compose renaming maps and so get a category $K_\J^R(x,y)$ with objects the fractions from $x$ to $y$ and arrows the renaming maps.

\end{definition}

As we shall see, we will also have a category with objects the $\J$-fractions and arrows the maps of fractions, and a functor including $K_\J^R(x,y)$ into this latter category. 
For now we will be content with giving the definition of the arrow component of this functor, without proving functoriality; namely, a renaming map $r$ from $(j,f)$ to $(k,g)$ as above is sent to the map $\iota(r)$ of fractions specified by the 2-arrow
\begin{equation}\label{eq:iota_on_arrows}
	\vcenter{\xymatrix{
		& u \ar[dd]_r \ar[dr]^f_{\ }="s2"\\
		 u\times_x v \ar[ur]_{\ }="s1"^{\pr_1} \ar[dr]^{\ }="t1"_{\pr_2} && y \\
		& v \ar[ur]_g^{\ }="t2"
		\ar@{=>}"s2";"t2"^{a_r}
		\ar@{=>}"s1";"t1"
	}}\,,
\end{equation}
where the 2-arrow on the left is the canonical lift of the identity 2-arrow $k\circ r \circ \pr_1 = k\circ \pr_2$ through the ff arrow $k$, 
using Lemma~\ref{lemma:unique_lifts}.

\begin{definition}

	The identity map $1\colon (j,f) \Rightarrow (j,f)$ on a $J$-fraction
	$x\xleftarrow{j} u \xrightarrow{f} y$ is given by $\iota(\id_u)$.

\end{definition}

The (vertical) composition of maps of $\J$-fractions proceeds as follows. Given

\begin{align*}
	t_1\colon & (j_1,f_1) \Rightarrow (j_2,f_2)\\
	t_2\colon & (j_2,f_2) \Rightarrow (j_3,f_3)
\end{align*}

where $x\xleftarrow{j_i} u_i \xrightarrow{f_i} b$, consider the 2-arrow $t_1 \oplus t_2$ filling the diagram\footnote{One might be concerned with the bracketing of the triple pullback here; for concreteness we can take $(u_1\times_x u_2)\times_x u_3$, it would not change the final result if we used the other option.}
\begin{equation}\label{eq:precomposition of maps of fractions}
	\vcenter{\xymatrix{
		&u_1\times_x u_2 \ar[r] \ar[]!DR;[dr] & u_1 
		\ar[dr]_(.2){\ }="s1"^{f_1} \\
		 u_1\times_x u_2\times_x u_3 \ar[]!UR;[ur]!DL \ar[]!DR;[dr]!UL && 
		 u_2 \ar[r]^(.1){\ }="t1"_(.1){\ }="s2"|{f_2} & y\\
		&u_2\times_x u_3 \ar[]!UR;[ur] \ar[r] & u_3 \ar[ur]^(.2){\ }="t2"_{f_3}
		 \ar@{=>}"s1";"t1"_{t_1}
		 \ar@{=>}"s2";"t2"_{t_2}
	}}\,,
\end{equation}
which we shall call the \emph{precomposition} of $t_1$ and $t_2$. 
We need to show that this 2-arrow descends along the arrow $u_1\times_x u_2\times_x u_3 \xrightarrow{\pr_{13}} u_1\times_x u_3 \in \J$. 
But $\pr_{13}$ is co-ff, and the source and target of $t_1 \oplus t_2$ factor as $u_1\times_x u_2\times_x u_3 \xrightarrow{\pr_{13}} u_1\times_x u_3 \to u_i \xrightarrow{f_i} y $ for $i=1$ and $i=3$ respectively. 
Thus $t_1 \oplus t_2$ descends uniquely, and we call this descended 2-arrow $t_1 + t_2$ (note that $+$ is \emph{not} a commutative operation!), and it gives a map of $\J$-fractions $(u_1,f_1) \Rightarrow (u_3,f_3)$.

\begin{remark}
	\label{rem:simple_vert_composition}

	If $u_1\times_x u_2 \times_x u_3 \to u_1\times_x u_3$ has a section, then the vertical composition $t_1 + t_2$ is the whiskering of $t_1\oplus t_2$ on the left with this section.

\end{remark}

An example of such a section arises when composing two maps of $\J$-fractions, where one of the maps arises from an inverible renaming map.
Here we say a renaming map with data $(r,a_r)$ as in Definition~\ref{def:renaming_map} is \emph{invertible} if $r$ is an invertible 1-arrow and $a_r$ is an invertible 2-arrow.
We record a special case of this as a lemma for future reference.

\begin{lemma}\label{lemma:renaming_iso_composite}
Let $x\xleftarrow{j} u \xrightarrow{f} y$ and $x\xleftarrow{k} v \xrightarrow{g} y$ be $\J$-fractions, $t\colon (j,f) \Rightarrow (k,g)$ be a map of fractions, and let $r\colon u' \to u$ be a renaming map from $x \xleftarrow{j'} u' \xrightarrow{f'} y$  to $(j,f)$, with $j' = jr$ and $a_r\colon f'\Rightarrow fr$.
Then $\iota(r)+t$ is given by the 2-arrow
\[
\hspace{1000pt minus 2000pt}
\diagram{
		&&u'\times_x u \ar[r]_{\ }="s1"^{\pr_1} \ar@/_/[dr]_{\pr_2}^{\ }="t1" & u'\ar[d]^r  \ar[dr]^{f'}_{\ }="s3"	\\
		u'\times_x v \ar[r]&u'\times_x u\times_x v  \ar[ur]^{\pr_{12}} \ar[dr]_{\pr_{23}} && u \ar[r]_(.1){\ }="s2"^(0.4){\ }="t3"|f & y\\
		&& u\times_x v \ar[ur]^{\pr_1} \ar[r]_{\pr_2} & v \ar[ur]^(.22){\ }="t2"_g
		\ar@{=>}"s2";"t2"_{t}
		\ar@{=>}"s1";"t1"^{\ell_{j\pr_2}}
		\ar@{=>}"s3";"t3"_{a_r}
	}
	\equals
	\diagram{
		&u'\ar[dr]^(.87){\ }="t1"_{r}   \ar@/^1pc/[drr]^{f'}_{\ }="s1"&	\\
		u'\times_x v  \ar[ur]^{\pr_{12}} \ar[dr]_{\pr_{23}} && u \ar[r]_(.1){\ }="s2"|f & y\\
		& u\times_x v \ar[ur]^{\pr_1} \ar[r]_{\pr_2} & v \ar[ur]^(.22){\ }="t2"_g
		\ar@{=>}"s2";"t2"_{t}
		\ar@{=>}"s1";"t1"_{a_r}
	}
\hspace{1000pt minus 2000pt}
\]
Hence if $a_r = \id$, $\iota(r)+t$ is given by the 2-arrow
\[
\hspace{1000pt minus 2000pt}
	\diagram{
		&&u'\times_x u \ar[r]_{\ }="s1"^{\pr_1} \ar@/_/[dr]_{\pr_2}^{\ }="t1" & u'\ar[d]^r  \ar[dr]^{f'}	\\
		u'\times_x v \ar[r]&u'\times_x u\times_x v  \ar[ur]^{\pr_{12}} \ar[dr]_{\pr_{23}} && u \ar[r]_(.1){\ }="s2"|f & y\\
		&& u\times_x v \ar[ur]^{\pr_1} \ar[r]_{\pr_2} & v \ar[ur]^(.2){\ }="t2"_g
		\ar@{=>}"s2";"t2"_{t}
		\ar@{=>}"s1";"t1"^{\ell_{j\pr_2}}
	}
	\equals
	\diagram{
		& &u \ar[dr]^f_{\ }="s"	 & \\
		u'\times_x v \ar[r]_{r\times \id} &u\times_x v \ar[ur]^{\pr_1} \ar[dr]_{\pr_2}& & y\\
		&   & v \ar[ur]^{\ }="t"_g
		\ar@{=>}"s";"t"_{t}
	}
\hspace{1000pt minus 2000pt}
\]
Further, given an \emph{invertible} renaming map $q$ from $(k,g)$ to $x\xleftarrow{k'} v' \xrightarrow{g'} y$, with $k=k'q$ and $g=g'q$ (hence $a_q=\id$), the composite $t+\iota(q)$ is given by the 2-arrow
\[
\hspace{1000pt minus 2000pt}
	\diagram{
		&&u\times_x v \ar[r]^{\pr_1} \ar[dr]_{\pr_2} & u \ar[dr]^{f}_(.2){\ }="s2"	\\
		u\times_x v' \ar[r]&u\times_x v\times_x v'  \ar[ur]^{\pr_{12}} \ar[dr]_{\pr_{23}} && v \ar[d]^q  \ar[r]^(.1){\ }="t2"_g & y\\
		&& v\times_x v' \ar@/^/[ur]^{\pr_1}_{\ }="s1" \ar[r]_{\pr_2}^{\ }="t1" & v' \ar[ur]_{g'}
		\ar@{=>}"s2";"t2"_{t}
		\ar@{=>}"s1";"t1"^{\ell_{k\pr_1}}
	}
	\equals
	\diagram{
		& &u \ar[dr]^f_{\ }="s"	\\
		u\times_x v' \ar[r]_{ \id\times q^{-1}} &u\times_x v \ar[ur]^{\pr_1} \ar[dr]_{\pr_2}& & y\\
		&   & v \ar[ur]^{\ }="t"_g
		\ar@{=>}"s";"t"_{t}
	}
\hspace{1000pt minus 2000pt}
\]
\end{lemma}

\begin{cor}
The assignment $\iota$ is a functor $K_\J^R(x,y) \to K_\J(x,y)$.
\end{cor}

\begin{proof}
Apply the first case in Lemma~\ref{lemma:renaming_iso_composite} to when $t = \iota(r')$ for an arbitrary renaming transformation $q$ and use Lemma~\ref{lemma:unique_lifts}.
\end{proof}

\begin{prop}

	We have a category $K_\J(x,y)$ with objects the $\J$-fractions from $x$ to $y$ and arrows the maps of $\J$-fractions.

\end{prop}

\begin{proof}
	That $1_{(j,f)}$ is the identity arrow for $x\xleftarrow{j}u\xrightarrow{f}y$ follows from the second and third cases of Lemma~\ref{lemma:renaming_iso_composite}, taking $r = \id_u$ and $q=\id_v$ respectively.

	We thus need only to show composition is associative. 
	Consider the diagram
	\begin{equation}\label{eq:unbiased_composition_three_2arrows}
		\raisebox{17ex}
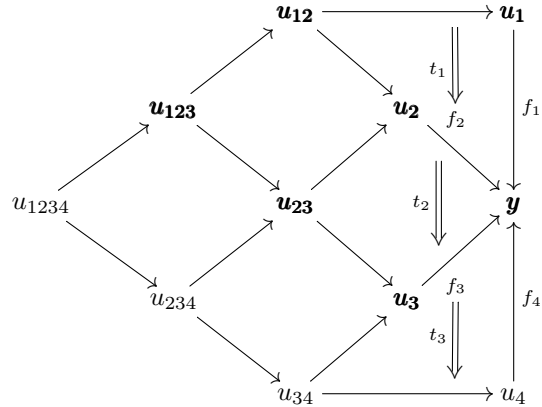
{\xymatrix{
			&& \pmb{u_{12}} \ar[rr]_(.73){\ }="s1" \ar[dr] && \pmb{u_1} \ar[dd]^{f_1}\\
			& \pmb{u_{123}} \ar[dr] \ar[ur] && \pmb{u_2} \ar[dr]^(.3){f_2}="t1"_(.4){\ }="s2" &\\
			u_{1234} \ar[ur] \ar[dr] && \pmb{u_{23}} \ar[ur] \ar[dr] && \pmb{y} \\
			& u_{234} \ar[ur] \ar[dr] && \pmb{u_3} \ar[ur]_(.3){f_3}="s3"^(.4){\ }="t2" &\\
			&& u_{34}\ar[rr]^(.73){\ }="t3" \ar[ur] && u_4 \ar[uu]_{f_4}
			\ar@{=>}"s1";"t1"_{t_1}
			\ar@{=>}"s2";"t2"_{t_2}
			\ar@{=>}"s3";"t3"_{t_3}
		}}
	\end{equation}
	where $u_{i\ldots k} := u_i\times_x \ldots \times_x u_k$. The bolded objects form a sub-diagram we will refer to below.
	We will show that the composites
	\[
		\xymatrix{
		u_{1234} \ar[r] & u_{14} \ar@/^1.2pc/[r]_{\ }="s"^{f_1}\ar@/_1.2pc/[r]^{\ }="t"_{f_4}
		& y \ar@{=>}"s";"t"^a
		}
	\]
	are equal to (\ref{eq:unbiased_composition_three_2arrows}) for  $a=(t_1+t_2)+t_3$ and $a = t_1 + (t_2 + t_3)$. 
	First consider $(t_1+t_2)+t_3$:
	\begin{align*}
		\xymatrix{
			u_{1234} \ar[r] & u_{14} \ar@/^1.7pc/[rr]_(.5){\ }="s" 
			\ar@/_1.7pc/[rr]^(.5){\ }="t" && y
			\ar@{=>}"s";"t"|(0.46){(t_1+t_2)+t_3\phantom{\big|}}
		}
		&\equals
		\xymatrix{
			u_{1234} \ar[r] & u_{134} \ar[r] & u_{14} 
			\ar@/^1.7pc/[rr]_(.5){\ }="s" 
			\ar@/_1.7pc/[rr]^(.5){\ }="t" && y
			\ar@{=>}"s";"t"|(0.46){(t_1+t_2)+t_3\phantom{\big|}}
		}\\
		&\\
		&\equals
		\vcenter{\xymatrix{
			&&& u_1 \ar[ddr]_(.4){\ }="s1" \\
			&& u_{13} \ar[ur] \ar[dr] \\
			u_{1234} \ar[r] & u_{134} \ar[ur] \ar[dr] && 
			u_3 \ar[r]^(.3){\ }="t1"_(.3){\ }="s2" & y\\
			&& u_{34} \ar[ur] \ar[dr]\\
			&&& u_4 \ar[uur]^(.4){\ }="t2"
			\ar@{=>}"s1";"t1"_{t_1+t_2}
			\ar@{=>}"s2";"t2"_{t_3}
		}}
	\end{align*}
	\begin{align*}
		\phantom{\xymatrix{
			u_{1234} \ar[r] & u_{14} \ar@/^1.5pc/[rr]_(.3){\ }="s" 
			\ar@/_1.5pc/[rr]^(.3){\ }="t" && y
			\ar@{=>}"s";"t"^{(t_1+t_2)+t_3}
		}}
		&\equals
		\vcenter{\xymatrix{
			& && \pmb{u_1} \ar[ddr]_(.4){\ }="s1" \\
			  & \pmb{u_{123}} \ar[r] & \pmb{u_{13}} \ar[ur] \ar[dr] && \\
			u_{1234} \ar[r] \ar[ur] & u_{134} \ar[dr] \ar[ur] && \pmb{u_3} \ar[r]^(.3){\ }="t1"_(.3){\ }="s2" & \pmb{y}\\
			&& u_{34} \ar[ur] \ar[dr] \\
			&&& u_4 \ar[uur]^(.4){\ }="t2"
			\ar@{=>}"s1";"t1"_{t_1+t_2}
			\ar@{=>}"s2";"t2"_{t_3}
		}}
	\end{align*}
	ommitting some of the labels on the 1-arrows for clarity. 
	Now the whiskered 2-arrow in the subdiagram on the bold symbols above is equal to the composite 2-arrow in the subdiagram of (\ref{eq:unbiased_composition_three_2arrows}) on the bold symbols, hence the whole diagram equals (\ref{eq:unbiased_composition_three_2arrows}). 
	A symmetric argument shows that $t_1 + (t_2 + t_3) \circ 1_{u_{1234} \to u_{14}}$ is also equal to (\ref{eq:unbiased_composition_three_2arrows}). 
	By uniqueness of descent, composition of maps of $\J$-fractions is associative, and $K_\J(x,y)$ is a category.
\end{proof}

\subsection{Defining the bicategory \texorpdfstring{$K_\J$}{K\_J}}

Now we want to show that $K_\J(x,y)$ is the hom-category of a bicategory, so we need a composition functor. 
Composing 1-arrows is easy:

\begin{definition}

	The composition of $\J$-fractions is the composite span
	
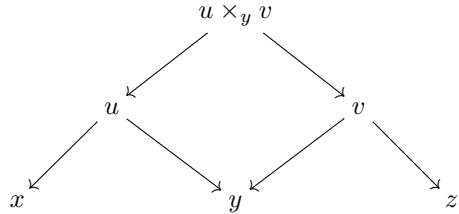
\[
		\xymatrix{
		&& u\times_y w \ar[dl] \ar[dr]\\
		& u \ar[dl] \ar[dr]&& w \ar[dl] \ar[dr] \\
		x && y && z
		}
	\]
	where recall we are assuming we have specified pullbacks of 1-arrows in $\J$, so this is well-defined.

\end{definition}

We shall define the composition in the bicategory $K_\J$ by defining left and right whiskering functors and proving the interchange law as outlined in \cite[pp 126-127]{Makkai}\footnote{Makkai says, helpfully, ``Next, we need to verify that thus we have defined functors \ldots we leave the task to the reader.'' [\emph{ibid.} page 127]} for the case where $K = \Cat$ and $\J$ is the class of fully faithful, surjective-on-objects functors.

\begin{definition}\label{def:right whiskering}
	
Let $t$ be a map of fractions from $x\leftarrow u\xrightarrow{f} y$ to $x\leftarrow v\xrightarrow{g} y$.
	The \emph{right whiskering} of $t$ by the $\J$-fraction $y\xleftarrow{l} w \xrightarrow{h} z$ is given by
	\[
		\xymatrix{
			&&& w \times_{y,f} u \ar[]!DR;[dr]^{f\circ \pr_2}_(.3){\ }="s"\\
			x&\ar[l]u\times_x v & \ar[l] 
				w\times_{y,f}(u\times_x v)\times_{g,y} w
				\ar[]!UR;[ur]!DL \ar[]!DR;[dr]!UL && w \ar[r]^h & z\\
			&&& v \times_{g,y} w \ar[]!UR;[ur]_{g\circ \pr_1}^(.3){\ }="t"
			\ar@{=>}"s";"t"_{\rho_{(w,h)}t}
		}
	\]
	where the 2-arrow $\rho_{(w,h)}t\colon \pr_1 \Rightarrow \pr_4$  is the unique lift through $l\colon w\to y$ of
	\[
		\xymatrix{
			&& u \ar[dr]^f\\
			w\times_{y,f}(u\times_x v)\times_{g,y} w \ar[r] & u\times_x v
			\ar[]!UR;[ur]_(0.7){\ }="s"
			\ar[]!DR;[dr]^(0.7){\ }="t"
			&& y\;.\\
			&& v \ar[ur]_g
			\ar@{=>}"s";"t"^t
		}
	\]

\end{definition}

\begin{prop}
	\label{prop:right_whisker_functorial}

	Right whiskering with $y\xleftarrow{l} w \xrightarrow{h} z$ is a functor $K_\J(x,y) \to K_\J(x,z)$.

\end{prop}

\begin{proof}

	First, let us show right whiskering preserves identity 2-arrows. 
	That is, the horizontal composition of a pair of identity 2-arrows is the identity 2-arrow of the composition of the 1-arrows. 
	Let $x\xleftarrow{j}u \xrightarrow{f} y$ be a fraction and consider the right whiskering of the map $\id_{(j,f)}$ by $(l,h)$. 
	This is the map of fractions given by
	\begin{equation}\label{eq:right_whiskered_identity}
		\vcenter{\xymatrix{
			x & \ar[l]
			w\times_y u^{[2]} \times_y w  
			\ar@/^.5pc/[]!UR;[r]!UL_{\ }="s"^{\pr_1} \ar@/_.5pc/[]!DR;[r]!DL^{\ }="t"_{\pr_4}
			& w \ar[r]^h & z
			\ar@{=>}"s";"t"
		}}
	\end{equation}
	where the 2-arrow is the unique lift of
	\[
		\xymatrix{
			w\times_y u^{[2]} \times_y w \ar[r] & u^{[2]} 
			\ar@/^/[]!UR;[r]!UL_{\ }="s"^{\pr_1} 
			\ar@/_/[]!DR;[r]!DL^{\ }="t"_{\pr_2} &
			u \ar[r]^f & y\;,
			\ar@{=>}"s";"t"
		}
	\]
	the unlabelled maps being the obvious projections. 
	But we have the equality
	\[
		\vcenter{\xymatrix{
			w\times_y u^{[2]} \times_y w 
			\ar@/^/[]!UR;[r]!UL_{\ }="s"^{\pr_1} 
			\ar@/_/[]!DR;[r]!DL^{\ }="t"_{\pr_4} &
			w
			\ar@{=>}"s";"t"
		}}
		\equals
		\vcenter{\xymatrix{
			w\times_y u^{[2]} \times_y w 
			\ar@/^/[]!UR;[r]!UL_{\ }="s"^{\pr_{12}} 
			\ar@/_/[]!DR;[r]!DL^{\ }="t"_{\pr_{34}} &
			u\times_x y \ar[r]^-{f\circ \pr_1} & w
			\ar@{=>}"s";"t"
		}}\,
	\]
	hence (\ref{eq:right_whiskered_identity}) is 
	\[
		\xymatrix{
			x & \ar[l] 
			w\times_y u^{[2]} \times_y w 
			\ar@/^/[]!UR;[r]!UL_{\ }="s"^{\pr_{12}} 
			\ar@/_/[]!DR;[r]!DL^{\ }="t"_{\pr_{34}} 
			\ar@{=>}"s";"t" &
			u\times_y w \ar[r]^-{h\circ f \circ \pr_1}&
			z
		}
		\equals
		\id_{(l,h)\circ(j,f)}\,.
	\]
	Thus whiskering is unital. 

	Now to prove that right whiskering preserves composition we will again use uniqueness of descent, and prove equal a pair of 2-arrows with  0-source a cover of the 0-source of the 2-arrows we are interested in. 
	Without loss of generality, we can right whisker by the fraction $(l,\id) = y\xleftarrow{l} w \xrightarrow{\id} w$, as the component of the fraction pointing in the forward direction plays no substantial r\^ole in what is to follow.

	Consider the composable pair of maps of fractions given by the data
	\[
		\xymatrix{
			u_1\times_x u_2 
			\ar@/^.5pc/[]!UR;[r]!UL_{\ }="s1"^{f_1\circ \pr_1}
			\ar@/_.5pc/[]!DR;[r]!DL^{\ }="t1"_{f_2\circ \pr_2}
			\ar@{=>}"s1";"t1"^{a_1}
			& y
		} \quad \text{ and } \quad
		\xymatrix{
			 u_2\times_x u_3 
			\ar@/^.5pc/[]!UR;[r]!UL_{\ }="s2"^{f_2\circ \pr_1}
			\ar@/_.5pc/[]!DR;[r]!DL^{\ }="t2"_{f_3\circ \pr_2}
			\ar@{=>}"s2";"t2"^{a_2}
			& y
		}\,.
	\]
	Let $u_{123}:= u_1\times_x u_2\times_x u_3$ and similarly for $u_{12}$, $u_{23}$, and consider the diagram
	\[
		\vcenter
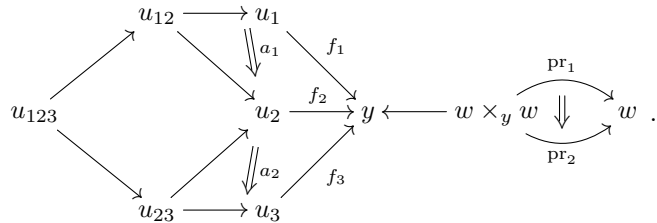
{\xymatrix{
			&u_{12} \ar[dr]^(.8){\ }="t1" \ar[r]_(.8){\ }="s1" & u_1 \ar[dr]^{f_1}\\
			u_{123} \ar[ur] \ar[dr] && u_2 \ar[r]^{f_2} & y & \ar[l] w\times_y w
			\ar@/^1pc/[r]_{\ }="s"^{\pr_1} 
			\ar@/_1pc/[r]^{\ }="t"_{\pr_2} 
			\ar@{=>}"s";"t" & w  \\
			&u_{23} \ar[ur]_(.8){\ }="s2" \ar[r]^(.8){\ }="t2" & u_3 \ar[ur]_{f_3}
			\ar@{=>}"s1";"t1"^{a_1}
			\ar@{=>}"s2";"t2"^{a_2}
		}}\,.
	\]
	We need to prove equal the pair of 2-arrows $(\rho_{(l,\id)} a_1) + (\rho_{(l,\id)} a_2)$ and $\rho_{(l,\id)} (a_1 + a_2)$ between the two 1-arrows $(w\times_y u_1)\times_x (u_3\times_y w) \simeq w\times_y u_{13} \times_y w \stackrel{\pr_i}{\longrightarrow} w$, for $i=1,3$.

	\begin{figure}[!t]
		\[
			\xymatrix{
				& 
				\text{\footnotesize{$w\times_y u_{12} \times_y w$ }}
				  \ar[dr]!UL
				  \ar[]!R;[drrr]!UL_(.37){\ }="s2"^(0.3){\pr_1} 
				  \ar[dddr]|!{[dd];[dr]!DL}\hole|!{[dd]!R;[drrr]}\hole \\
				w\times_y u_{123} \times_y w  
				  \ar[dddr] 
				  \ar[dddd]_p 
				  \ar[]!UR;[ur]!DL
				  \ar[]!DR;[dr]!UL &
				& 
				w\times_y u_2 \times_y w 
				  \ar[rr]^(0){\ }="t2"_(0){\ }="s3" &
				& 
				w 
				  \ar[ddd] \\
				& 
				\text{\large{$w\times_y u_{23} \times_y w$}}
				  \ar[ur]!DL
				  \ar[]!R;[urrr]!DL^(.37){\ }="t3"_(0.3){\pr_3}
				  \ar[dddr] & 
				 \\
				&
				& 
				\text{\footnotesize{$u_{12}$}} 
				  \ar[r]_(.8)\hole="s4" 
				  \ar[dr]^(.87)\hole="t4" & 
				\text{\footnotesize{$u_1$}} 
				  \ar[dr]|{f_1} \\
				& 
				u_{123} 
				  \ar[ur]|!{[uu];[dr]}\hole 
				  \ar[dr] 
				  \ar[ddr] &
				& 
				u_2 
				  \ar[r]|{f_2} & 
				y 
				  \ar@{=}[dd] \\
				w\times_y u_{13} \times_y w 
				  \ar[drr]!L 
				  \ar@{-->}@/_2.5pc/[]!R;[rrrruuuu]!DL_(0.2){\pr_3}^(0.17)\hole="t6" 
				  \ar@{-->}@/^1pc/[rrrruuuu]!DL^(0.2){\pr_1}_(0.22)\hole="s6" &
				& 
				\text{\large{$u_{23}$}}
				  \ar[ur]_(.87)\hole="s5" 
				  \ar[r]^(.9)\hole="t5" & 
				\text{\large{$u_3$}}
				  \ar[ur]|{f_3}\\
				&
				&
				u_{13} 
				  \ar@/^1pc/[rr]_(.35){\ }="s1" 
				  \ar@/_1pc/[rr]^(.35){\ }="t1" &
				& 
				y
				\ar@{=>}"s1";"t1"^{\,a_1+a_2}
				\ar@{=>}"s2";"t2"^{\rho_{(l,\id)} a_1}
				\ar@{=>}"s3";"t3"^{\rho_{(l,\id)} a_2}
				\ar@{=>}"s4";"t4"^{a_1}
				\ar@{=>}"s5";"t5"^{a_2}
				\ar@{=>}"s6";"t6"_(0.35){(\ast)}
			}
		\]
		\caption{Right whiskering is functorial}
		\label{fig:right_whiskering_functorial}
		\end{figure}

		In Figure \ref{fig:right_whiskering_functorial} the sub-diagram consisting of just the solid arrows together with the 2-arrows between them 2-commutes, so the precomposition $(\rho_{(l,\id)} a_1) \oplus (\rho_{(l,\id)} a_2)$ is given by the top layer of the diagram, namely
		\[
			\xymatrix{
				& 
				\text{\footnotesize{$w\times_y u_{12} \times_y w$ }}
				  \ar[dr]!UL
				  \ar[]!R;[drrr]!UL_(.37){\ }="s2"^(0.3){\pr_1} 
				   \\
				w\times_y u_{123} \times_y w  
				  \ar[]!UR;[ur]!DL
				  \ar[]!DR;[dr]!UL &
				& 
				w\times_y u_2 \times_y w 
				  \ar[rr]^(0){\ }="t2"_(0){\ }="s3" &
				& 
				w  \\
				& 
				\text{\large{$w\times_y u_{23} \times_y w$}}
				  \ar[ur]!DL
				  \ar[]!R;[urrr]!DL^(.37){\ }="t3"_(0.3){\pr_3}
				   & 
				\ar@{=>}"s2";"t2"^(0.65){\rho_{(l,\id)} a_1}
				\ar@{=>}"s3";"t3"^{\rho_{(l,\id)} a_2}
				 }
			\] 
		and $(\rho_{(l,\id)} a_1) + (\rho_{(l,\id)} a_2)$ is given by the unique descent of this 2-arrow along $p$.
		The 2-arrow marked $(\ast)$ is the whiskering $\rho_{(l,\id)} (a_1 + a_2)$, and forms a 2-commuting diagram with $a_1+a_2$ and the 1-arrows $w\times_y u_{13} \times_y w \to u_{13}$ and $w\to y$. 
		The 2-cell
		\[
			\xymatrix{
				w\times_y u_{123} \times_y w \ar[r] & 
				w\times_y u_{13} \times_y w 
				\ar@{-->}@/^1.3pc/[]!UR;[rr]_(0.4)\hole="s"^{\pr_1}
				\ar@{-->}@/_1.3pc/[]!DR;[rr]^(0.4)\hole="t"_{\pr_3} &&
				w
				\ar@{=>}"s";"t"^{(\ast)}
			}
		\]
		in Figure~\ref{fig:right_whiskering_functorial} is, by uniqueness of lifts through $p$ and $w\to y$ (both in $\J$) equal to $(\rho_{(l,\id)} a_1) \oplus (\rho_{(l,\id)} a_2)$. 
		Thus the descent of $(\rho_{(l,\id)} a_1) \oplus (\rho_{(l,\id)} a_2)$ along $p$ is just $\rho_{(l,\id)} (a_1 + a_2)$, which is what we needed to prove.
\end{proof}

The definition of the left whiskering is slightly more complicated, as it is such that it doesn't permit us to nearly ignore half of the span as we can for right whiskering. 
What we shall do is define left whiskering by a general $\J$-fraction $x\xleftarrow{j}u\xrightarrow{f}y$ in two cases, using the factorisation $(j,f) = (\id_u,f)\circ (j,\id_u)$.

\subsubsection{Case I: left whiskering by $(\id_u,f)$}

Let $a$ be a map of fractions from $y\leftarrow v_1 \xrightarrow{g} z$ to $y\leftarrow v_2 \xrightarrow{h} z$, and  $f\colon u\to y$ an arrow in $K$. 
The whiskered 2-arrow will be a map of fractions from $u \xleftarrow{\pr_1} u\times_y v_1 \xrightarrow{g\circ \pr_2} z$ to $u \xleftarrow{\pr_1} u\times_y v_2 \xrightarrow{h\circ \pr_2} z$, and so the desired 2-arrow in $K$ will be of the form 
\[
	\xymatrix{
			u\times_y v_{12} 
			\ar@/^.5pc/[]!UR;[r]!UL_{\ }="s1"^{g\circ \pr_1}
			\ar@/_.5pc/[]!DR;[r]!DL^{\ }="t1"_{h\circ \pr_2}
			\ar@{=>}"s1";"t1"
			& z
		}
\]
 where $v_{12} := v_1 \times_y v_2$.

\begin{definition}\label{def:left whiskerin case I}

	The left whiskering of the map $a$ of $\J$-fractions by the fraction $u\xleftarrow{\id_u} u \xrightarrow{f}y$ is given by the 2-arrow $\lambda^I_{(\id_u,f)} a$, defined as
	\[
		\xymatrix{
			 u\times_y v_{12} \ar[r] & v_{12} \ar@/^1pc/[rr]_{\ }="s" \ar@/_1pc/[rr]^{\ }="t" && z
			\ar@{=>}"s";"t"_a
		}\;.
	\]
\end{definition}

\subsubsection{Case II: left whiskering by $(j,\id_u)$}

We have $v_i\to y=u$ $\J$-covers for $i=1,2$, and now here $v_{12} := v_1\times_u v_2$.
We also let $V_{12} := v_1 \times_x v_2$, and there is a canonical map $v_{12}\to V_{12}$ fitting into a commutative diagram
\[
	\xymatrix{
		v_{12} \ar[r]\ar[d] & V_{12} \ar[d]\\
		u \ar[r]_j & x
	}
\]
where the left, bottom and right arrows are all in $\J$. 
Thus from Lemma~\ref{lemma:stronger_co-ff} we have that the top arrow is co-ff.
Notice also that there is a trivial factorisation of $\pr_i\colon v_{12} \to v_i$ as $v_{12} \to V_{12} \xrightarrow{\pr_i} v_i$.

\begin{definition}\label{def:left whiskering case II}
The left whiskering of the map $a$ by $x\xleftarrow{j}u \xrightarrow{\id} u$ is given by the 2-arrow $\lambda^{II}_{(j,\id)}a$ in $K$ defined via unique descent along the co-ff arrow $v_{12}\times_x V_{12} \to V_{12}$ by the equation
\[
	\diagram{
		v_{12}\times_x V_{12} \ar[r]^-{\pr_2} & 
		V_{12} 
			\ar@/^1.5pc/[rr]_(.65){\ }="s"^{g\circ \pr_1} 
			\ar@/_1.5pc/[rr]^(.65){\ }="t"_{h\circ \pr_2} && 
		z \ar@{=>}"s";"t"_{\lambda^{II}_{(j,\id)}a}
	}
	\equals
	\diagram{
		&& V_{12} \ar[dr]_{\ }="s2"^{g\circ \pr_1} \\
		v_{12}\times_x V_{12} \ar@/^/[urr]^{\pr_2}_{\ }="s1" 
			\ar@/_/[drr]_{\pr_2}^{\ }="t3" \ar[r]^(.85){\ }="t1"_(.85){\ }="s3" & 
			v_{12} \ar[ur] \ar[dr] && z\\
		&& V_{12} \ar[ur]^{\ }="t2"_{h\circ\pr_2}
		\ar@{=>}"s1";"t1"
		\ar@{=>}"s2";"t2"_a
		\ar@{=>}"s3";"t3"
	}\;.
\]
\end{definition}

Left whiskering by an arbitrary fraction $x\xleftarrow{j} u \xrightarrow{f} y$ will then be the composite of the two (putative) functors given by cases I and II.

\begin{prop}\label{prop:left_whisker_functorial}

	Left whiskering with $x\xleftarrow{j} u \xrightarrow{f} y$ is a functor $K_{\J}(y,z) \to K_{\J}(x,z)$.

\end{prop}

The proof that left whiskering preserves (vertical) composition will be deferred to appendix \ref{app:left_whiskering_functorial}, as it is a sizable calculation.

\begin{proof}{(Left whiskering is unital)}
	We want to do the whiskering 
	\[
		\xymatrix{
			x & \ar[l]_j u \ar[r]^f & y & \ar[l] 
			v\times_y v 
			  \ar@/^.5pc/[]!UR;[r]!UL_{\ }="s" 
			  \ar@/_.5pc/[]!DR;[r]!DL^{\ }="t" 
			& v \ar[r]^g &z
			\ar@{=>}"s";"t"^a
		}\,.
	\]
	Note that without loss of generality we can assume $g=\id_v$, the general case follows exactly the same argument merely with $g$ right whiskered onto all the 2-cells involved. 
	We treat case I and case II of the definition of left whiskering separately.
	\begin{itemize}
	\item[Case I.] Note that $v_{12}$ in this case is $v\times_y v$.
	The left whiskering of the identity map on $y\leftarrow v \xrightarrow{\id} v$ has 2-cell component
	\[
		\xymatrix{
			u\times_y (v\times_y v) \ar[r] & v \times_y v \ar@/^.5pc/[]!UR;[r]!UL_{\ }="s" \ar@/_.5pc/[]!DR;[r]!DL^{\ }="t" & v
			\ar@{=>}"s";"t"^a
		}
	\]
	but $u\times_y (v\times_y v)\simeq (u\times_y v) \times_u (u\times_y v)$, and by Lemma~\ref{lemma:unique_lifts} this is equal to
	\[
		\xymatrix{
			(u\times_y v) \times_u (u\times_y v) \ar@/^.5pc/[]!UR;[r]!UL_{\ }="s" \ar@/_.5pc/[]!DR;[r]!DL^{\ }="t" & u \times_y v  \ar[r] & v
			\ar@{=>}"s";"t"^a
		}
	\]
	and this is the identity map on the composite $u \leftarrow u\times_y v \to v$, as required.

	\item[Case II.] Again, in this case, $v_{12} = v\times_u v$, which for now will be denoted $v^{[2]}$ and $V_{12} = v\times_x v$.
	Recall that the 2-cell component of the whiskered identity map will be the unique 2-cell $\lambda :=\lambda^{II}_{(j,\id)}a$ in the diagram
	\[
	\hspace{1000pt minus 2000pt}
	\diagram{
		v^{[2]}\times_x (v\times_x v) \ar[r] & v\times_x v \ar@/^1.5pc/[r]_(.5){\ }="s"^{\pr_1} \ar@/_1.5pc/[r]^(.5){\ }="t"_{\pr_2} & v \ar@{=>}"s";"t"^{\lambda}
	}
	\equals
	\diagram{
		&& v\times_x v \ar[dr]_{\ }="s2"^{\pr_1} \\
		v^{[2]}\times_x (v\times_x v) \ar@/^/[urr]_{\ }="s1" 
			\ar@/_/[drr]^{\ }="t3" \ar[r]^(.85){\ }="t1"_(.85){\ }="s3" & 
			v^{[2]} \ar[ur] \ar[dr] && v\\
		&& v\times_x v \ar[ur]^{\ }="t2"_{\pr_2}
		\ar@{=>}"s1";"t1"
		\ar@{=>}"s2";"t2"
		\ar@{=>}"s3";"t3"
	}\;,
	\hspace{1000pt minus 2000pt}
	\]
	which exists by descent along the $\J$-cover $v^{[2]}\times_x (v\times_x v) \to v\times_x v$. 
	However, by Lemma~\ref{lemma:unique_lifts} the canonical 2-arrow $\pr_1 \Rightarrow \pr_2\colon v\times_x v\to v$ fits into such an equation of 2-arrows, and this is none other than the 2-cell component of the identity 2-arrow on the composite fraction $x\xleftarrow{j} v \xrightarrow{\id} v$. 

	\end{itemize}

	Putting case I and case II together, we have that left whiskering $\lambda_{(j,f)}\colon K_\J(y,z) \to K_\J(x,z)$ preserves identity maps.
\end{proof}

With Propositions \ref{prop:left_whisker_functorial} and \ref{prop:right_whisker_functorial} we can define a composition functor.

\begin{lemma}
Left and right whiskering fit together to give a functor
\[
	K_\J(x,y)\times K_\J(y,z) \to K_\J(x,z).
\]
\end{lemma}
\begin{proof}
The only thing that remains to check is that middle-four interchange holds, as per the hypothesis of \cite[Proposition~II.3.1]{CWM}. 
This proof is deferred to Appendix~\ref{app:middle_four}.
\end{proof}

In order for this to be the composition functor for a bicategory we just need to now show that it is coherently associative and unital. 
In fact, by virtue of Definition~\ref{def:strict_pretopology}, this composition is \emph{strictly} unital, since the composition of any fraction with the identity fraction of its source or target is unchanged.
We define the left and right \emph{unitors} to be the appropriate identity 2-arrows.

\begin{lemma}
	The unitors are natural.
\end{lemma}

\begin{proof}
Recall that the unitors are meant to be natural transformations as in the diagram (of categories)
\[
	\xymatrix{
		K_\J(x,y) \ar@{^(->}[r] \ar[dr]_=^{\ }="t" & K_\J(x,y) \times K_\J(y,y) \ar[d]_(0.2){\ }="s" \\
		& K_\J(x,y)
		\ar@{=>}"s";"t"^{l_y}
	}
	\qquad
	\xymatrix{
		K_\J(x,y) \ar@{^(->}[r] \ar[dr]_=^{\ }="t" & K_\J(x,x) \times K_\J(x,y) \ar[d]_(0.2){\ }="s"  \\
		& K_\J(x,y)
		\ar@{=>}"s";"t"^{r_x}
	}
\]
Our unitors consist of identity arrows, so we need to prove that these diagrams commute \emph{on the nose}.
This is true at the level of objects, so we just need to check that the appropriate wiskerings of an arbitrary map of $\J$-fractions (i.e.\ an arrow in $K_\J(x,y)$ by identity fractions result in the original map of fractions.

In the case of $r_x$, we can apply Defintion~\ref{def:left whiskerin case I}, whiskering by the fraction $(\id_x,\id_x)$.
But then we just get the original map of $\J$-spans. 
Thus the left triangle above commutes and $r_x$ is natural.

In the case of $l_y$, we use Definition~\ref{def:right whiskering}, with $y\xleftarrow{l}w\xrightarrow{h} z$ being $(\id_y,\id_y)$.
But then $w\times_{y,f}(u\times_x v)\times_{g,y} w = u\times_x v$, and we are lifting through an identity arrow, and then whiskering (in $K$) with $\id_y$.
The result is then the original map of $\J$-spans, making the right triangle above commute and so $l_y$ is natural.
\end{proof}

\begin{definition}\label{def:associator}

	The \emph{associator} for the 3-tuple of composable fractions
	\[
		\xymatrix{
			& u \ar[dl]_j \ar[dr]^f&& v \ar[dl]_k \ar[dr]^g && w \ar[dl]_l \ar[dr]^h\\
			x_1 && x_2 && x_3 && x_4
		}
	\]
	is the invertible map $\iota(a_{uvw})\colon \big[(l,h)\circ(k,g)\big]\circ(j,f) \Rightarrow (l,h)\circ\big[(k,g)\circ(j,f)\big]$ of $\J$-fractions associated to the renaming map arising from the canonical isomorphism
	\[
		a_{uvw}\colon u\times_{x_2} (v \times_{x_3} w) \xrightarrow{\simeq} (u \times_{x_2} v) \times_{x_3} w
	\]
	over $x_1$, together the appropriate identity 2-arrow.

\end{definition}

\begin{lemma}\label{lemma:associator_is_natural}
	The associator 2-arrow in Definition~\ref{def:associator} is natural.
\end{lemma}

\begin{proof}
This is proved in Appendix~\ref{app:assoc_natural}.
\end{proof}

We can check that the associator satisfies the necessary coherence diagrams in the bicategory of fractions and renaming maps, since it will then hold in the bicategory of fractions and maps of fractions.
In fact, since the renaming map in question is the associator for products in the strict slice $K/x_1$ (i.e.\ strict pullbacks in $K$), it satisfies coherence by the universal property of pullbacks.

\begin{remark}

	If we do not assume that pullbacks of identity arrows are again identity
	arrows, then we do get nontrivial unitors, but they are, like the
	associator, renaming maps, and one can check they are coherent.

\end{remark}

We have thus proved:\footnote{cf Bartels, who says ``The various coherence conditions in a (weak) 2-category are now tedious but straightforward to check.'' \cite{Bartels}}

\begin{prop}

	There is a bicategory $K_\J$ with the same objects as $K$, fractions as 1-arrows and maps of fractions as 2-arrows.

\end{prop}

We now define an identity-on-objects strict 2-functor $A_\J\colon K \to K_\J$ as follows. 
For a 1-arrow $f\colon x\to y$ of $K$, let $A_\J(f)$ be the fraction $x \xleftarrow{\id_x} x \xrightarrow{f} y$. Given a 2-arrow $a\colon f \Rightarrow g\colon x\to y$ in $K$, let  $A_\J(a)$ be the map of fractions $(\id_x,f) \Rightarrow (\id_x,g)$ determined by
\[
	\xymatrix{
		x\times_x x = x
		\ar@/^.5pc/[]!UR;[r]!UL_{\ }="s"^f \ar@/_.5pc/[]!DR;[r]!DL^{\ }="t"_g
		& y
		\ar@{=>}"s";"t"^a
	}\, ,
\]
To check that $A_\J$ is a strict 2-functor, we need to check first that it is functorial for vertical composition of 2-arrows.
In the definition of vertical composition of 2-cells, the diagram~(\ref{eq:precomposition of maps of fractions}) in the case of maps of fractions in the image of $A_\J$ collapses as all objects $u_i$ and their fibre products reduce to $x$, with all arrows between them identity arrows.
The descended 2-arrow is then just the vertical composite in $K$, and so $A_\J$ preserves vertical composition. It is also simple to show that $A_\J$ preserves identity 2-arrows.

Secondly, we need to show that $A_\J$ is functorial for horizontal composition. 
Identity 1-arrows are preserved strictly, as is composition of 1-arrows, so it is just a matter of checking that horizontal composition of 2-cells is preserved.
Since horizontal composition is defined via left and right whiskering, we need to check that whiskering a map of fractions in the image of $A_\J$ by a fraction in the image of $A_\J$ is of the same form.
The right whiskering of $A_\J(a\colon f_1\Rightarrow f_2)$ where $f_1,f_2\colon x\to y$ by $y\xleftarrow{\id_y}y\xrightarrow{g} z$ involves a 2-cell $\rho_{(\id,g)}a$ (see Definition~\ref{def:right whiskering}).
Since our fractions are in the image of $A_\J$, the diagram again collapses so that all appearances of $u\times_x v$ are equal to $x$, and $w=y$, so that $\rho_{(\id,g)}a = a$, and the final result has the 2-cell component the right whiskering of $a$ by $g$.
The left whiskering we need is case I, so we consider Definition~\ref{def:left whiskerin case I}. 
Consider the map of fractions $A_\J(a\colon g_1\Rightarrow g_2)$ where $g_1,g_2\colon y\to z$ and whisker it by $x\xleftarrow{\id_x}x\xrightarrow{f}y$.
Now in the definition of the 2-cell $\lambda^I_{(\id,f)}a$, we have $v_{12} = v_1 = v_2 = y$, the maps between them are identity maps, $u=x$, and $u\times_y v_{12} \to v_{12}$ is just $f$.
Thus the whiskered map of fractions is again in the image of $A_\J$, and we have proved that $A_\J$ is a strict 2-functor.

\begin{lemma}\label{lemma:AJ locally fully faithful}

	The 2-functor $A_\J$ is locally fully faithful, that is, $K(x,y) \to K_\J(x,y)$ is fully faithful for all objects $x$ and $y$ of $K$.

\end{lemma}

\begin{proof}
	A map of $\J$-fractions $(\id_x,f) \Rightarrow (\id_x,g)$ is precisely the same data as a 2-arrow $f\Rightarrow g$ in $K$.
\end{proof}

\begin{definition}

	Given $\J$, a 1-arrow in $q\colon x\to y$ in $K$ is \emph{$\J$-locally split} if there is an arrow $j\colon u\to y$ in $\J$ and a diagram of the form
	\[
		\vcenter{\xymatrix{
			& x\ar[d]^q \\
			u \ar[r]^(.7){\ }="t"_j \ar[ur]_(.7){\ }="s"^s & y 
			\ar@{=>}"s";"t"
		}}
	\]
	with the 2-arrow invertible.
	A 1-arrow in $K$ is a \emph{weak equivalence} if it is ff and $\J$-locally split. 
	Denote the class of weak equivalences by $W_\J$.

\end{definition}

Clearly $\J \subset W_\J$ as we are assuming all arrows in $\J$ are ff, and every arrow in $\J$ is trivially $\J$-locally split.

\begin{prop}
	\label{prop:characterising_arrows_inverted_by_A}

	Let $f$ be a 1-arrow of $K$. Then $A_\J(f)$ an equivalence if and only if $f\in W_\J$.

\end{prop}

\begin{proof}

	First assume $f\colon x\to y$ is in $\J$; we will show $(\id_x,f)$ is an equivalence, with quasi-inverse $(f,\id_x)$. 
	This is because $(\id_x,f)\circ (f,\id_x)=(f,f)$, which is isomorphic to $(\id_y,\id_y)$ by the invertible map of fractions
	\[
		\xymatrix{
			&&  x \ar[dr]^f_{\ }="s" \\
			y&\ar[l]_f x \ar[ur]^{\id_x} \ar[dr]_f &  & y \\
			&&  y \ar[ur]_{\id_y}^{\ }="t"
			\ar@{=>}"s";"t"_{\id}
		}
	\]
	In the other direction, $(f,\id_y)\circ(\id_y,f) = (\pr_1,\pr_2)$, where $\pr_i\colon x\times_y x \to x$, $i=1,2$, are the projections (both of which are in $\J$).
	There is the canonical invertible 2-cell $\ell_f\colon \pr_1\Rightarrow \pr_2$, which gives an isomorphism of $\J$-fractions
	\[
		\xymatrix{
			&&  x\times_y x \ar[dr]^{\pr_2}_{\ }="s" \\
			x&\ar[l]_{\pr_1} x\times_y x \ar[ur]^{\id} \ar[dr]_{\pr_1} &  & x \\
			&&  x \ar[ur]_{\id_x}^{\ }="t"
			\ar@{=>}"s";"t"_{\ell_f}
		}
	\]

	In the other direction, let $f\colon x\to y$ be a 1-arrow in $K$ such that $(\id_x,f)$ is an equivalence in $K_\J$, i.e.~there is a $\J$-fraction $y\xleftarrow{j}u\xrightarrow{g} x$ such that
	\begin{enumerate}
		\item[1.] $(\id_x,f)\circ(j,g) \xrightarrow{\sim} (\id_y,\id_y)$
		\item[2.] $(\id_x,\id_x) \xrightarrow{\sim}(j,g)\circ(\id_x,f)$
	\end{enumerate}

	Point 1 implies that we have an isomorphism of $\J$-fractions
	\[
		\xymatrix{
			&&  u \ar[dr]^{f\circ g}_{\ }="s" \\
			y&\ar[l]_j u \ar[ur]^{\id_u} \ar[dr]_j &  & y \\
			&&  y \ar[ur]_{\id_y}^{\ }="t"
			\ar@{=>}"s";"t"_{\simeq}
		}
	\]
	The right hand half of this diagram means that $f$ is $\J$-locally split. 
	Since $(\id_x,f)$ is an equivalence it is ff in the bicategory $K_\J$. 
	Then as $A_\J$ is locally fully faithful it reflects ff 1-arrows, hence $f$ is ff in $K$.
	Thus $f$ is both $\J$-locally split and ff, hence is in $W_\J$.
\end{proof}

For a number of diverse examples of weak equivalences in various categories of internal categories and groupoids, see \cite[\S 8]{Roberts_12}.

\subsection{\texorpdfstring{$K_\J$}{K\_J} as a localisation}

Given a 2-category (or bicategory) $B$ with a class $W$ of 1-arrows, we say that a 2-functor $Q\colon B \to \widetilde{B}$ is a \emph{localisation of $B$ at $W$} if it sends the 1-arrows in $W$ to equivalences in $\widetilde{B}$ and is universal with this property. 
This latter means that for any bicategory $A$ precomposition with $Q$,
\[
	Q^* \colon \Bicat(\widetilde{B},A) \to \Bicat_W(B,A),
\]
is an equivalence of hom-bicategories, with $\Bicat_W$ meaning the full sub-bicategory on those 2-functors sending arrows in $W$ to equivalences.

\begin{theorem}
	\label{thm:anafunctors_localise}

	A 2-site $(K,\J)$ admits a bicategory of fractions for $W_\J$, and the inclusion 2-functor $A_J\colon K\to K_\J$ is a localisation at the class $W_\J$ of weak equivalences.

\end{theorem}

\begin{proof}
That $(K,\J)$ admits a bicategory of fractions for $W_\J$ is \cite[Theorem 6]{Roberts_14a} (the weaker hypotheses there on 2-sites are implied by the ones here).
The proof that $A_\J$ is a localisation proceeds via Pronk's comparison theorem \cite[Proposition~24]{Pronk_96}, the conditions of which imply that the canonical 2-functor $K[W_\J^{-1}] \to K_\J$ is an equivalence of bicategories.
Here $K[W_\J^{-1}]$ is the bicategory of fractions constructed by Pronk, and we recall the conditions of the comparison theorem for ease of reference, using the current notation:
\begin{enumerate}
	\item[EF1.]
               $A_\J$ is essentially surjective,
       \item[EF2.]
               For every 1-arrow $f$ of $K_\J$ there are 1-arrows $w \in W_\J$ and $g$ of $K$ such that
               $A_\J(g) \stackrel{\sim}{\Rightarrow} f \circ A_\J(w)$,
       \item[EF3.]
               $A_\J$ is locally fully faithful.
\end{enumerate}
We now show these conditions hold. 
To begin with, the 2-functor $A_\J$ sends weak equivalences to equivalences by Proposition~\ref{prop:characterising_arrows_inverted_by_A}.

	\begin{enumerate}[EF1.]
		\item
			$A_\J$ is the identity on objects, and hence surjective on objects.

		\item
			This is equivalent to showing that for any $\J$-fraction $x\xleftarrow{j} u\xrightarrow{f} y$ there are 1-arrows $w,g$ in $K$ such that $w$ is in $W_\J$ and
			\[
				(j,f) \stackrel{\sim}{\Rightarrow} A_\J(g)\circ\overline{A_\J(w)}
			\]
			where $\overline{A_\J(w)}$ is some pseudoinverse for $A_\J(w)$. 
			We can take $w=j$ and $g=f$, since by the proof of Proposition~\ref{prop:characterising_arrows_inverted_by_A}, $(j,\id_u)$ is a pseudoinverse for $(\id_u,j)$, and the composite fraction of $(j,\id_u)$ and $(\id_u,f)$ is just $(j,f)$.

		\item
			This holds by Lemma~\ref{lemma:AJ locally fully faithful}.
	\end{enumerate}

\noindent Thus $A_\J$ is a localisation of $K$ at $W_\J$.
\end{proof}

As a last remark, one would like to know if the localisation of $K$ at the weak equivalences is locally essentially small. 
This can be assured by the following result, where we have used the condition WISC from \cite{Roberts_12}, which states that every object $x$ of $K$ has a set of covers that are weakly initial in the subcategory of $K/x$ on the $\J$-covers.

\begin{prop}
	\label{prop:WISC_gives_locally_small}

	If the locally essentially small 2-site $(K,\J)$ satisfies WISC, then $K_\J$ is locally essentially small, and hence so is any localisation of $K$ at $W_\J$.

\end{prop}

Notice that local essential smallness in \emph{not} automatic, as there are well-pointed toposes with a natural numbers object, otherwise very nice categories, for which the 2-category of internal categories fails the hypothesis of Proposition~\ref{prop:WISC_gives_locally_small}. 
For example the toposes of material sets in models of ZF as given by Gitik (see \cite{vdBerg-Moerdijk_14}) and Karagila \cite{Karagila_14}, or the well-pointed topos of structural sets arising from \cite{Roberts_15a}.
Karagila has also described an explicit model of ZF in which the category of anafunctors from the discrete groupoid $\mathbb{N}$ to the one-object groupoid $\mathbf{B}(\mathbb{Z}/2)$ is not essentially small.\footnote{See the answers to the MathOverflow question \url{https://mathoverflow.net/q/264585/}. User `aws' also gave a model of ZFA---choiceless Zermelo--Fraenkel set theory with atoms---with the same property.}

Finally, note that nothing in this paper relies on $K$ being a (2,1)-category, namely a 2-category with only invertible 2-arrows. 
This is usually assumed for results subsumed by Theorem~\ref{thm:anafunctors_localise}, but is unnecessary in the framework presented here. 
The following example will be treated in a forthcoming paper.

\begin{example}
Take a pretopos $E$ with stable reflexive coequalisers, and define $K$ to be the wide, locally full sub-2-category of $\Cat(E)^{op}$ taking only those internal functors whose object component is a coproduct inclusion, which we shall call a \emph{complemented cofibration}.
Let $\J$ consist of the class $W$ of complemented cofibrations $f\colon X\to Y$ that are ff and essentially surjective (i.e.\ $X_0\times_{f_0,Y_0,s} Y_1 \xrightarrow{t\pr_2} Y_0$ is a regular epimorphism).
Then $W$  consists of ff and co-ff arrows, contains all identity arrows, and is closed under composition and---most crucially---pushout. 
This makes $(\Cat(E)^{op},W)$ a 2-site as defined in this paper, and the constructions here involving $W$-fractions in $\Cat(E)^{op}$ correspond to analogous dual constructions involving \emph{left} $W$-fractions (certain cospans) in $\Cat(E)$.
The analogous result holds for $\Gpd(E)$ in place of $\Cat(E)$, and in this case $E$ can be an arbitrary pretopos.
\end{example}

This gives a bicategorical perspective on a generalisation of the case of small \emph{groupoids}, studied in \cite{LandNikolausSzumilo} using cofibration categories as a presentation of $(\infty,1)$-categories.

\appendix

\section{Proof that left whiskering in \texorpdfstring{$K_\J$}{K\_J} 
preserves vertical composition}
	\label{app:left_whiskering_functorial}


The definition of left whiskering in $K_\J$ is slightly more complicated, as it is
such that it doesn't permit us to ignore half of the span as we
can for right whiskering. Revall that we define left whiskering by a general $\J$-fraction $x\xleftarrow{j}u\xrightarrow{f} y$ in two cases, using the factorisation $(j,f) = (\id_u,f)\circ (j,\id_u)$.

\subsection{Case I: left whiskering by $(\id_u,f)$}


Recall the definition of case I of left whiskering (Definition~\ref{def:left whiskerin case I}).

\begin{definition}

	The left whiskering of $a\colon (j,g) \Rightarrow (k,h)$ by $(\id_u,f)$ is given by the 2-arrow $\lambda^I_{(\id,f)} a$, defined as
	\[
		\xymatrix{
			u\times_y v_{12} \ar[r] & v_{12} \ar@/^1pc/[rr]_{\ }="s" \ar@/_1pc/[rr]^{\ }="t" && z
			\ar@{=>}"s";"t"_a
		}\;.
	\]
\end{definition}

We now show left whisking by $(\id,f)$ preserves composition.
In the following, let $\lambda^I (-) = \lambda^I_{(\id,f)} (-)$

\begin{align*}
	\diagram{
		u\times_yv_{123} \ar[dd] \\
		\\
		u\times_yv_{13} \ar@/^1.5pc/[rrr]_(.65){\ }="s" \ar@/_1.5pc/[rrr]^(.65){\ }="t" &&& z\\
		\\
		\ar@{=>}"s";"t"_{\lambda^I(a_1+a_2)}
	}
	& \equals 
	\diagram{
		u\times_yv_{123} \ar[dd] \\
		\\
		u\times_yv_{13} \ar[dd]\\
		\\
		v_{13} \ar@/^1.5pc/[rrr]_(.55){\ }="s" \ar@/_1.5pc/[rrr]^(.55){\ }="t" &&& z
		\ar@{=>}"s";"t"_{a_1+a_2}
	}\\ 
	& \equals 
	\diagram{
		u\times_yv_{123} \ar[dd] \\
		\\
		v_{123} \ar[dd]\\
		\\
		v_{13} \ar@/^1.5pc/[rrr]_(.55){\ }="s" \ar@/_1.5pc/[rrr]^(.55){\ }="t" &&& z
		\ar@{=>}"s";"t"_{a_1+a_2}
	}\\ 
	& \equals 
	\diagram{
		u\times_y v_{123} \ar[dd] \\
		& v_{12} \ar[dr]_{\ }="s1" \ar[r]& v_1 \ar[dr]_(.15){\ }="s2"&\\
		v_{123} \ar[ur] \ar[dr] && v_2 \ar[r]^(.1){\ }="t2"_(.1){\ }="s3" & z\\
		& v_{23} \ar[ur]^{\ }="t1" \ar[r] &v_3 \ar[ur]^(.15){\ }="t3"&
		\ar@{=>}"s2";"t2"_{a_1}
		\ar@{=>}"s3";"t3"_{a_2}\\
		\\
	}\\ 
	& \equals
	\diagram{
		&u\times_y v_{12} \ar[dd] \ar[r]  \ar[dr]& u\times_y v_1 \ar[dr]_(.2){\ }="s1"\\
		u\times_y v_{123} \ar[dd] \ar[ur] && u\times_y v_2 \ar[dd] \ar[r]^(.1){\ }="t1" & z 
		\ar@{=}[dd]\\
		&v_{12} \ar[dr]&&\\
		v_{123} \ar[ur] \ar[dr] & & v_2 \ar[r]_(.1){\ }="s2" & z\\
		&v_{23} \ar[r] \ar[ur]& v_3\ar[ur]^(.2){\ }="t2" &
		\ar@{=>}"s1";"t1"_{\lambda^I a_1}
		\ar@{=>}"s2";"t2"_{a_2}
	}\\
	& \equals
	\diagram{
		&u\times_y v_{12} \ar[r]  \ar[dr]& u\times_y v_1 \ar[dr]_(.2){\ }="s1"\\
		u\times_y v_{123} \ar[ur] \ar[dr] && u\times_y v_2  \ar[r]^(.1){\ }="t1"_(.1){\ }="s2" & z\\ 
		& u\times_y v_{23} \ar[ur] \ar[r] & u\times_y v_2 \ar[ur]^(.2){\ }="t2" &
		\ar@{=>}"s1";"t1"_{\lambda^I a_1}
		\ar@{=>}"s2";"t2"_{\lambda^I a_2}
	}\\ 
	& \equals
	\vcenter{\xymatrix{
		u\times_y v_{123} \ar[dd] \\
		\\
		u\times_y v_{13} \ar@/^1.5pc/[rrr]_(.65){\ }="s" \ar@/_1.5pc/[rrr]^(.65){\ }="t" &&& z
		\ar@{=>}"s";"t"_{\lambda^I a_1+\lambda^I a_2}
	}}
\end{align*}

By uniqueness of descent, $\lambda^I(a_1+a_2) = \lambda^I a_1 + \lambda^I
a_2$.

 \subsection{Case II: left whiskering by $(j,\id_u)$}

Recall the notations $V_{12} := v_1 \times_x v_2$,  $V_{ij\ldots} := v_i\times_x v_j \times_x \ldots$. and the canonical maps $v_{ij\ldots} \to V_{ij\ldots}$. 
Recall the definition of case II of left whiskering (Definition \ref{def:left whiskering case II}):

\begin{definition}
The left whiskering of $a\colon (j,g) \Rightarrow (k,h)$ by $(j,\id_u)$  is given by the 2-arrow $\lambda^{II}_{(j,\id)}a$ in $K$ defined via unique descent by the equation
\[
	\diagram{
		v_{12}\times_x V_{12} \ar[r]^-{\pr_2} & V_{12} \ar@/^1.5pc/[rr]_(.65){\ }="s" \ar@/_1.5pc/[rr]^(.65){\ }="t" && z \ar@{=>}"s";"t"_{\lambda_{(j,\id)}a}
	}
	\equals
	\diagram{
		&& V_{12} \ar[dr]_{\ }="s2" \\
		v_{12}\times_x V_{12} \ar@/^/[urr]^{\pr_2}_{\ }="s1" 
			\ar@/_/[drr]_{\pr_2}^{\ }="t3" \ar[r]^(.85){\ }="t1"_(.85){\ }="s3" & 
			v_{12} \ar[ur] \ar[dr] && z\\
		&& V_{12} \ar[ur]^{\ }="t2"%
		\ar@{=>}"s1";"t1"
		\ar@{=>}"s2";"t2"_a
		\ar@{=>}"s3";"t3"
	}\;.
\]
\end{definition}

We now prove left whiskering by $(j,\id_u)$ preserves composition. In the following, let $\lambda^{II}(-) :=\lambda^{II}_{(j,\id)}(-)$
\begin{align*}
	\diagram@C-8pt{
		v_{123}\times_x V_{123} \ar[dd]\\
		\\
		v_{13}\times_x V_{13}\ar[dd]\\
		\\
		V_{13} \ar@/^1.7pc/[rr]_(.7){\ }="s" \ar@/_1.5pc/[rr]^(.7){\ }="t" &&z
		\ar@{=>}"s";"t"_(0.35){\tiny \lambda^{II}(a_1+a_2)\!}
	}
	& \equals 
	\diagram@C-8pt{
		v_{123}\times_x V_{123} \ar[dd]\\
		&& V_{13} \ar[dr]_{\ }="s2" &\\
		v_{13}\times_x \text{\footnotesize{$V_{13}$}} \ar[r]^(.8){\ }="t1"_(.8){\ }="s3" \ar@/^/[urr]_{\ }="s1" \ar@/_/[drr]^{\ }="t3" & 
			v_{13} \ar[ur] \ar[dr] && z\\
		&& \text{\large{$V_{13}$}} \ar[ur]^{\ }="t2"&\\
		\ar@{=>}"s2";"t2"_{a_1+a_2}
		\ar@{=>}"s1";"t1"
		\ar@{=>}"s3";"t3"
	}\\
	& \equals
	\diagram@C-8pt{
		&& \text{\footnotesize{$V_{123}$}} \ar@/_1.2pc/[ddd]!UL|(0.56)\hole|(0.7)\hole \\
		v_{123}\times_x V_{123} \ar@/^/[urr]_{\ }="s1" \ar@/_/[drr]^{\ }="t2" \ar[r]^(.8){\ }="t1"_(.8){\ }="s2" & v_{123} \ar[ddd]|(0.25)\hole \ar[ur]!DL \ar[dr]\\
		&&\text{\large{$V_{123}$}} \ar@/^1pc/[ddd]!UR &\\
		&& V_{13} \ar[dr]|(0.42)\hole &\\
		& v_{13}  \ar[ur]_(.8){\ }="s3" \ar[dr]^(.8){\ }="t3" && z\\
		&&\text{\large{$V_{13}$}} \ar[]!R;[ur]&
		\ar@{=>}"s3";"t3"_{a_1+a_2}
		\ar@{=>}"s1";"t1"
		\ar@{=>}"s2";"t2"
	}\\
	&\equals
	\diagram@C-8pt{
		&& V_{123} \ar[dr]_{\ }="s3"\\
		v_{123}\times_x V_{123} \ar@/^/[urr]_{\ }="s1" \ar@/_/[drr]^{\ }="t2" \ar[r]^(.8){\ }="t1"_(.8){\ }="s2" & v_{123}  \ar[ur]!DL \ar[dr]
		&& z\\
		&&V_{123}  \ar[ur]^{\ }="t3"&
		\ar@{=>}"s3";"t3"_{a_1\oplus a_2}
		\ar@{=>}"s1";"t1"
		\ar@{=>}"s2";"t2"
	}\\
	& \equals
	\diagram@C-8pt{
		&V_{123} \ar[rr] &&V_{12} \ar[ddrr]_{\ }="s3"\\
		&& v_{12} \ar[ur] \ar[dr] &  \\
		v_{123}\times_x V_{123} \ar[uur]_{\ }="s1" \ar[r]^(.6){\ }="t1"_(.6){\ }="s2" \ar[ddr]^{\ }="t2"
		& v_{123} \ar[uu] \ar[dd] \ar[dr] \ar[ur] && v_2 \ar[rr]^{\ }="t3"_{\ }="s4"&  & z\\
		&&v_{23} \ar[dr] \ar[ur] &  \\
		&V_{123} \ar[rr] && V_{23} \ar[uurr]^{\ }="t4"
		\ar@{=>}"s1";"t1"
		\ar@{=>}"s2";"t2"
		\ar@{=>}"s3";"t3"_{a_1}
		\ar@{=>}"s4";"t4"_{a_2}
	}\\
	& \equals
	\diagram@C-8pt{
		&V_{123} \ar[rr] &&V_{12} \ar[ddrr]_{\ }="s3" \\
		&& \pmb{v_{12}} \ar[ur] \ar[r] & \pmb{V_{12}} \ar[dr]^{\ }="t3" \\
		\pmb{v_{123}\times_x V_{123}} \ar[uur]_{\ }="s1" \ar[r]^(.6){\ }="t1"_(.6){\ }="s2" \ar[ddr]^{\ }="t2"
		& \pmb{v_{123}} \ar[uu] \ar[dd] \ar[dr] \ar[ur] &&& \pmb{v_2} \ar[r]& z\\
		&& \pmb{v_{23}} \ar[dr] \ar[r] & \pmb{V_{23}} \ar[ur]_{\ }="s4" \\
		&V_{123} \ar[rr] && V_{23} \ar[uurr]^{\ }="t4"
		\ar@{=>}"s1";"t1"
		\ar@{=>}"s2";"t2"
		\ar@{=>}"s3";"t3"_{a_1}
		\ar@{=>}"s4";"t4"_{a_2}
	}\\
	& \equals
	\diagram@C-8pt{
		&& &V_{12} \ar[dddrr]_{\ }="s5" \\
		& \pmb{v_{12}\times_x V_{12}} \ar@/^/[urr]_{\ }="s1" \ar[r]^(.8){\ }="t1"_(.8){\ }="s2" \ar@/_/[drr]^{\ }="t2" & \pmb{v_{12}} \ar[ur] \ar[dr]\\
		&&& \pmb{V_{12}} \ar[dr]^{\ }="t5"&\\
		\pmb{v_{123}\times_x V_{123}} \ar[uur] \ar[ddr] &&&& \pmb{v_2} \ar[r] & z \\
		&&& \pmb{V_{23}} \ar[ur]_{\ }="s6" &\\
		& \pmb{v_{23} \times_x V_{23}} \ar@/^/[urr]_{\ }="s3" \ar[r]^(.8){\ }="t3"_(.8){\ }="s4" \ar@/_/[drr]^{\ }="t4" & \pmb{v_{23}} \ar[ur] \ar[dr]\\
		&& &V_{23} \ar[uuurr]^{\ }="t6"
		\ar@{=>}"s1";"t1"
		\ar@{=>}"s2";"t2"
		\ar@{=>}"s3";"t3"
		\ar@{=>}"s4";"t4"
		\ar@{=>}"s5";"t5"_{a_1}
		\ar@{=>}"s6";"t6"_{a_2}
	}\\
	\intertext{(where the subdiagrams on the bold symbols are equal)}
	& \equals
	
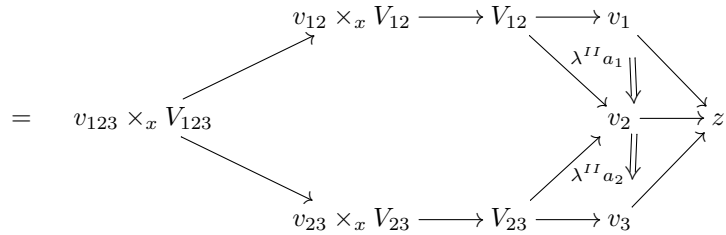
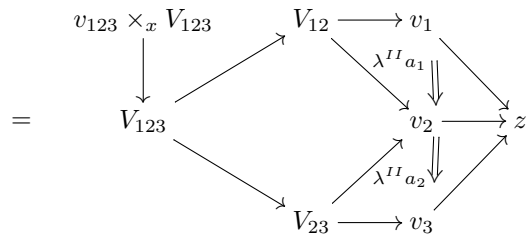
\diagram@C-8pt{
		& v_{12} \times_x V_{12} \ar[r] & V_{12} \ar[dr] \ar[r] & v_1 \ar[dr]_(.25){\ }="s1"\\
		v_{123} \times_x V_{123} \ar[ur] \ar[dr] && & v_2 \ar[r]^(.15){\ }="t1"_(.15){\ }="s2" & z \\
		& v_{23}\times_x V_{23} \ar[r] & V_{23} \ar[ur] \ar[r] & v_3 \ar[ur]^(.25){\ }="t2"
		\ar@{=>}"s1";"t1"_(.01){\tiny \lambda^{II} a_1\!}
		\ar@{=>}"s2";"t2"_(.99){\tiny \lambda^{II} a_2}
	}\\
	& \equals
	\diagram@C-8pt{
		v_{123}\times_x V_{123} \ar[d] & V_{12} \ar[r] \ar[dr] & v_1 \ar[dr]_(.25){\ }="s1" \\
		V_{123} \ar[ur] \ar[dr] && v_2 \ar[r]^(.15){\ }="t1"_(.15){\ }="s2" & z\\
		& V_{23} \ar[ur] \ar[r] & v_3 \ar[ur]^(.25){\ }="t2"
		\ar@{=>}"s1";"t1"_(.01){\tiny \lambda^{II} a_1\!}
		\ar@{=>}"s2";"t2"_(.99){\tiny \lambda^{II} a_2}
	}\\
	&\equals
	\diagram@C-8pt{
		v_{123}\times_x V_{123} \ar[d] \\
		V_{123} \ar[d]\\
		V_{13} \ar@/^1.5pc/[rrr]_(.65){\ }="s" \ar@/_1.5pc/[rrr]^(.65){\ }="t" &&&z
		\ar@{=>}"s";"t"_{\lambda^{II} a_1+\lambda^{II} a_2}
	}
\hspace{1000pt minus 2000pt}
\end{align*}

By uniqueness of descent, we have $\lambda^{II}(a_1+a_2) = \lambda^{II} a_1+\lambda^{II} a_2$.

Putting the two results this appendix together, arbitrary left whiskering preserves vertical composition.

\section{Proof that the associator is natural}\label{app:assoc_natural}

We will use \cite[Proposition~II.3.2]{CWM}, which says the naturality condition for a transformation between functors $A\times B \times C \to D$ can be checked in each component separately. 
Thus we only need to show the putative associator for $K_\J$ is natural for 2-arrows arising from double right whiskering, double left whiskering and left+right whiskering: 
\[
	\xymatrix{
		\ar@/^1pc/[r]_{\ }="s" \ar@/_1pc/[r]^{\ }="t" & \ar[r]& \ar[r]&
		\ar@{=>}"s";"t"
	}
\qquad
	\xymatrix{
		\ar[r]&  \ar[r] & \ar@/^1pc/[r]_{\ }="s" \ar@/_1pc/[r]^{\ }="t" &
		\ar@{=>}"s";"t"
	}
\qquad
	\xymatrix{
		\ar[r]& \ar@/^1pc/[r]_{\ }="s" \ar@/_1pc/[r]^{\ }="t" &  \ar[r]&
		\ar@{=>}"s";"t"
	}
\]
where these 1-arrows are $\J$-fractions, and the 2-arrows are maps of $\J$-fractions.

The major tool here is Lemma~\ref{lemma:renaming_iso_composite}, in particular the second and third cases, since the associators arise from invertible renaming transformations.

For the rest of this appendix, fix composable triples of $\J$-fractions:
\begin{align*}
	x_1 \xleftarrow{j_1} u_1 \xrightarrow{f_1} x_2\quad 
	x_2 \xleftarrow{j_2} u_2 \xrightarrow{f_2} x_3\quad 
	x_3 \xleftarrow{j_3} u_3 \xrightarrow{f_3} x_4
	\\
	x_1 \xleftarrow{k_1} v_1 \xrightarrow{g_1} x_2\quad 
	x_2 \xleftarrow{k_2} v_2 \xrightarrow{g_2} x_3\quad 
	x_3 \xleftarrow{k_3} v_3 \xrightarrow{g_3} x_4
\end{align*}
and maps of fractions $a_i\colon (j_i,f_i) \Rightarrow (k_i,g_i)$ for $i=1,2,3$.
Each step will only use one of $a_1,a_2$ or $a_3$ at a time.

We also need some notation for pulled back arrows, else there will be a confusing proliferation of projection maps.
So, given $f,g$ in the pullback square below (in our given 2-category $K$), the projection maps will be denoted $\widetilde{f}$ and $\widetilde{\jmath}$ as shown:
\[
	\xymatrix{
		x\times_y z \ar[r]^-{\widetilde{\jmath}} \ar[d]_{\widetilde{g}} & z \ar[d]^g \\
		x \ar[r]_j & y
	}
\]
If we need to pull an arrow $g$ back along two different maps $x_1 \xrightarrow{j_1} y$ and $x_2 \xrightarrow{j_2} y$, the results will be denoted $\widetilde{g}^{j_1}$ and $\widetilde{g}^{j_2}$ respectively.
To save space, horizontal composition will be denoted by juxtaposition (in function composition order), and the identity 2-arrow in $K$ on a 1-arrow $f$ will be denoted $1_f$.
 
\subsection{Double right whiskered}

See Definition~\ref{def:right whiskering} for how right whiskering is defined.  
We need to both right whisker $a_1$ twice in succession (by $(j_2,f_2)$ and $(j_3,f_3)$), and also whisker it by $(j_3,f_3)\circ (j_2,f_2) = (j_2\widetilde{j}_3,f_3\widetilde{f}_2)$.
The results then need to be vertically composed with associators in the appropriate order and compared.
Namely, we need to show the following 2-arrows are equal:
\[
	\rho_{(j_2\widetilde{\jmath_3},f_3\widetilde{f}_2)}a_1 + \iota(a_{v_1,u_2,u_3}) 
	\stackrel{?}{=} \iota(a_{u_1,u_2,u_3}) + \rho_{(j_3,f_3)}(\rho_{(j_2,f_2)}a_1).
\]
Or rather, we will show 
\[
	(\iota(a_{u_1,u_2,u_3}))^{-1} + \rho_{(j_2\widetilde{\jmath_3},f_3\widetilde{f}_2)}a_1 + \iota(a_{v_1,u_2,u_3}) \stackrel{?}{=} \rho_{(j_3,f_3)}(\rho_{(j_2,f_2)}a_1).
\]
Denote the 2-arrow data of $\rho_{(j_2\widetilde{\jmath_3},f_3\widetilde{f}_2)}a_1$ by $\rho_{f_3\tilde{f}_2}a$ and the 2-arrow data of $\rho_{(j_3,f_3)}(\rho_{(j_2,f_2)}a_1)$ by $\rho_{f_3}(\rho_{f_2}a)$.
Then $\rho_{f_3}(\rho_{f_2}a) := 1_{f_3}(1_{f_2}\widetilde{a}_1^{j_2})^\sim$, as shown the diagram in Figure~\ref{fig:double right whisker},
\begin{figure}
\[
\xymatrix@R3ex{
	& (u_1u_2)u_3 \ar[drr]^{(f_2\widetilde{f_1})^\sim}_{\ }="s1"\\
	[(u_1u_2)u_3]\times_{x_1} [(v_1u_2)u_3] \ar[ur] \ar[dr] \ar[ddd]_{q}
	&&& u_3 \ar[r]^{f_3} \ar[ddd]^{j_3} &x_4 \\
	&(v_1u_2)u_3 \ar[urr]_{(f_2\widetilde{g_1})^\sim}^{\ }="t1" &&\\
	& u_1u_2 \ar[dr]^{\widetilde{f_1}}_{\ }="s2"\\
	u_1u_2\times_{x_1} v_1u_2 \ar[ur] \ar[dr] \ar[ddd]_{p} && u_2 \ar[r]_{f_2} \ar[ddd]^{j_2} & x_3\\
	& v_1u_2 \ar[ur]_{\widetilde{g_1}}^{\ }="t2"\\
	& u_1 \ar[dr]^{f_1}_{\ }="s3"\\
	u_1\times_{x_1} v_1 \ar[ur] \ar[dr] && x_2\\
	& v_1 \ar[ur]_{g_1}^{\ }="t3"
	\ar@{=>}"s1";"t1"_{(1_{f_2}\widetilde{a}_1^{j_2})^\sim}
	\ar@{=>}"s2";"t2"_{\widetilde{a}_1^{j_2}}
	\ar@{=>}"s3";"t3"_{a_1}
}
\]
\caption{Constructing the double right whiskering}
\label{fig:double right whisker}
\end{figure}
where $(u_1u_2)u_3 := (u_1\times_{x_2} u_2)\times_{x_3} u_3$, $(v_1u_2)u_3:=(v_1\times_{x_2} u_2)\times_{x_3} u_3$ and similarly for $u_1u_2$ and $v_1u_2$.
The 2-arrows $\widetilde{a}_1^{j_2}$ and $(1_{f_2}\widetilde{a}_1^{j_2})^\sim$ are the unique 2-arrows satisfying
\begin{align}
1_{j_3}(1_{f_2}\widetilde{a}_1^{j_2})^\sim & = 1_{f_2}\widetilde{a}_1^{j_2}1_q \label{eq:left_assoc_eq_1}\\
1_{j_2}\widetilde{a}_1^{j_2} & = a_1 1_p \label{eq:left_assoc_eq_2}
\end{align}
respectively, using the fact $j_2$ and $j_3$ are ff.
We also have $\rho_{f_3\tilde{f}_2}a := 1_{f_3}1_{f_2}\widetilde{a}_1^{j_2\widetilde{\jmath_3}}$, as shown in the diagram in Figure~\ref{fig:right whisker composite}.
\begin{figure}
\[
\xymatrix@R3ex{
	& u_1(u_2u_3) \ar[dr]^{\widetilde{f_1}^{j_2\widetilde{\jmath_3}}}_{\ }="s1"\\
	[u_1(u_2u_3)]\times_{x_1} [v_1(u_2u_3)]\ar[ur] \ar[dr] \ar[dddd]_{\pi}&& u_2u_3 \ar[r]^{\widetilde{f_2}} \ar[dd]^{\widetilde{\jmath_3}} & u_3 \ar[r]^{f_3} \ar[dd]^{j_3} & x_4\\
	& v_1(u_2u_3) \ar[ur]_{\widetilde{g_1}^{j_2\widetilde{\jmath_3}}}^{\ }="t1" \\
	&& u_2\ar[r]_{f_2} \ar[dd]^{j_2}&x_3\\
	& u_1 \ar[dr]^{f_1}_{\ }="s2"\\
	u_1\times_{x_1} v_1 \ar[ur] \ar[dr] && x_2\\
	& v_1 \ar[ur]_{g_1}^{\ }="t2"
	\ar@{=>}"s1";"t1"_{\widetilde{a}_1^{j_2\widetilde{\jmath_3}}}
	\ar@{=>}"s2";"t2"_{a_1}
}
\]
\caption{Constructing the right whiskering by the composite}
\label{fig:right whisker composite}
\end{figure}
Again, using the fact $j_2\widetilde{\jmath_3}$ is ff, $\widetilde{a}_1^{j_2\widetilde{\jmath_3}}$ is the unique 2-arrow satisfying
\begin{equation}\label{eq:left_assoc_eq_3}
1_{j_2}1_{\widetilde{\jmath_3}}\widetilde{a}_1^{j_2\widetilde{\jmath_3}} = a_11_\pi
\end{equation}
Finally, we note that we have
\begin{equation}\label{eq:left_assoc_eq_4}
\pi \alpha = pq
\end{equation}
where $\alpha:= (a_{u_1,u_2,u_3}\times a_{v_1,u_2,u_3})^{-1}$.
Using Lemma~\ref{lemma:renaming_iso_composite}, we see that $(\iota(a_{u_1,u_2,u_3}))^{-1} + \rho_{(j_2\widetilde{\jmath_3},f_3\widetilde{f}_2)}a_1 + \iota(a_{v_1,u_2,u_3})$ is given by
\[
	\xymatrix{
		[(u_1u_2)u_3]\times_{x_1} [(v_1u_2)u_3] 
		\ar[r]^{\alpha} & [u_1(u_2u_3)]\times_{x_1} [v_1(u_2u_3)] \ar@/^1pc/!UR;[rr]_(.7){\ }="s" \ar@/_1pc/!DR;[rr]^(.7){\ }="t" && u_3 \ar[r]^{f_3\widetilde{f_2}} & x_4
		\ar@{=>}"s";"t"_{\widetilde{a}_1^{j_2\widetilde{\jmath_3}}}
	}
\]
We can thus calculate
\begin{align*}
1_{j_2}1_{\widetilde{\jmath_3}}\widetilde{a}_1^{j_2\widetilde{\jmath_3}} 1_\alpha &
= a_1 1_\pi 1_\alpha \qquad \text{using \eqref{eq:left_assoc_eq_3}}\\
& = a_1 1_p 1_q  \qquad \text{using \eqref{eq:left_assoc_eq_4}} \\
& = 1_{j_2}\widetilde{a}_1^{j_2} 1_q  \qquad \text{using \eqref{eq:left_assoc_eq_2}}
\end{align*}
and hence, since $j_2$ is ff, we have $1_{\widetilde{\jmath_3}}\widetilde{a}_1^{j_2\widetilde{\jmath_3}} 1_\alpha = \widetilde{a}_1^{j_2} 1_q$.
Continuing,
\begin{align*}
1_{j_3}1_{\widetilde{f_2}}\widetilde{a}_1^{j_2\widetilde{\jmath_3}} 1_\alpha
& = 1_{f_2}1_{\widetilde{\jmath_3}}\widetilde{a}_1^{j_2\widetilde{\jmath_3}} 1_\alpha \qquad \text{by definition of $\widetilde{f}_2$ and $\widetilde{\jmath_3}$}\\
& = 1_{j_3} \widetilde{a}_1^{j_2} 1_q\\
& = 1_{j_3}(1_{f_2}\widetilde{a}_1^{j_2})^\sim \qquad \text{using \eqref{eq:left_assoc_eq_1}},
\end{align*}
but now since $j_3$ is ff, we have $1_{\widetilde{f_2}}\widetilde{a}_1^{j_2\widetilde{\jmath_3}} 1_\alpha = (1_{f_2}\widetilde{a}_1^{j_2})\widetilde{\phantom{m}}$, and hence 
\begin{align*}
(\iota(a_{u_1,u_2,u_3}))^{-1} + \rho_{(j_2\widetilde{\jmath_3},f_3\widetilde{f}_2)}a_1 + \iota(a_{v_1,u_2,u_3}) & = 1_{f_3}1_{\widetilde{f_2}}\widetilde{a}_1^{j_2\widetilde{\jmath_3}} 1_\alpha\\
& = 1_{f_3} (1_{f_2}\widetilde{a}_1^{j_2})^\sim\\
& = \rho_{(j_3,f_3)}(\rho_{(j_2,f_2)}a_1),
\end{align*}
which is what we needed to show.

\subsection{Double left whiskered}

See Definitions~\ref{def:left whiskerin case I} and \ref{def:left whiskering case II} for how left whiskering is defined.  
We need to both left whisker $a_3$ twice in succession (by $(j_2,f_2)$ and $(j_1,f_1)$), and also whisker it by $(j_2,f_2)\circ (j_1,f_1)$.
The results then need to be vertically composed with associators in the appropriate order and compared.
Namely, we need to show the following 2-arrows are equal:
\[
	\lambda_{(j_1,f_1)}(\lambda_{(j_2,f_2)}a_3) + \iota(a_{u_1,u_2,v_3}) 
\stackrel{?}{=} 
	\iota(a_{u_1,u_2,u_3}) + \lambda_{(j_1\widetilde{\jmath_2},f_2\widetilde{f}_1)}a_3
\]
What we will show is the following
\[
	 \lambda_{(j_1,f_1)}(\lambda_{(j_2,f_2)}a_3) \stackrel{?}{=} 
	 \iota(a_{u_1,u_2,u_3}) + \lambda_{(j_1\widetilde{\jmath_2},f_2\widetilde{f}_1)}a_3 + \iota(a_{u_1,u_2,v_3})^{-1}.
\]

We can give a direct construction of the left whiskering, combining the two cases given in Definitions~\ref{def:left whiskerin case I} and \ref{def:left whiskering case II}, as follows.\footnote{This construction was found during revision of the article, after referee comments. 
I thank Fosco Loregian for the whimsical suggestion to call it the ``teapot lemma''.}

\begin{lemma}\label{lemma:left_whisker_combined}
Let $y\xleftarrow {j}v_1 \xrightarrow{f} z$ and $y\xleftarrow {k}v_2 \xrightarrow{g} z$ be $\J$-fractions. The left whiskering of the map $a\colon (j,f) \Rightarrow (k,g)$ by $x\xleftarrow{l} u\xrightarrow{h} y$ is given by the 2-arrow $\lambda a$ that is the unique descent of $a\circ 1_{\widetilde{h}^f\times\widetilde{h}^g}$ along the co-ff arrow $i$ in the following diagram
\[
	\xymatrix{
	&& u\times_y v_1 \ar[dr]^{f\widetilde{h}^f}_{\ }="s2" 
	\ar@/_1.5pc/[]!DL;[ddd]!UL |(0.54)\hole_(0.75){\pr_2}
	\\
	(u\times_y v_1)\times_u (u\times_y v_2) \ar[r]^-{i} \ar@/_1pc/[dddr]^{\widetilde{h}^f\times\widetilde{h}^g} & 
	(u\times_y v_1)\times_x (u\times_y v_2)\ar[];[ur]!L^{\pr_{12}} \ar[dr]_{\pr_{34}} && z \ar@{=}[ddd]\\
	&& u\times_y v_2 \ar[ur]_{g\widetilde{h}^g}^{\ }="t2" \ar@/^1.5pc/[ddd]_(0.75){\pr_2}&\\
	&& v_1\ar[dr]^f|(0.35)\hole&\\
	&v_1\times_y v_2 \ar[ur]_{\ }="s1"^{\pr_1} \ar[dr]^{\ }="t1"_{\pr_2} && z\\
	&& v_2\ar[]!R;[ur]_g
	\ar@{=>}"s1";"t1"^a
	\ar@{=>}"s2";"t2"_{\lambda a}
	}
\]
\end{lemma}
\begin{proof}
We first need to verify that $i$ is indeed co-ff. 
But this follows from the fact it fits into a commutative square
\[
	\xymatrix{
		(u\times_y v_1)\times_u (u\times_y v_2) \ar[r]^-{i} \ar[d] & 
		(u\times_y v_1)\times_x (u\times_y v_2) \ar[d]\\
		u \ar[r]_l & x
	}
\]
where the left and right vertical arrows are in $\J$ (and $l\in\J$ by assumption), hence $i$ is a map between $\J$-covers, and we can apply Lemma~\ref{lemma:stronger_co-ff}.

We thus need to verify that the left whiskerings (in $K$) of $\lambda_{(l,h)} a$ and $\lambda a$ with $i$ are equal, since then $i$ being co-ff means $\lambda_{(l,h)} a = \lambda a$.
But this follows from the definition of the left whiskering of $\lambda_{(\id_u,h)}a$ with $(l,\id_u)$.
\end{proof}

We first construct the 2-arrow $\lambda_{f_1}(\lambda_{f_2}a_3)$ specifying $\lambda_{(j_1,f_1)}(\lambda_{(j_2,f_2)}a_3)$, using the diagram in Figure~\ref{fig:double left whisker}.
\begin{figure}
\[
	\noindent\makebox[\textwidth]{%
	\xymatrix{
		[u_1(u_2u_3)]\times_{u_1\times_{x_1}u_2} [u_1(u_2v_3)]  \ar[rr]^-{P} \ar[dd]_{Q} \ar[rd] &
		& (u_2u_3)\times_{u_2}(u_2v_3) \ar[d]_{\widetilde{\jmath_3}} \ar[r]_-{q} \ar[rdd]_(.7){i_q}& u_3v_3 \ar[d]^{j_3} 
		\ar@/^1pc/[rr]_{\ }="s1"^{f_3\pr_1} \ar@/_1pc/[rr]^{\ }="t1"_{g_3\pr_2} && x_4 
		\\
		& 
		u_1\times_{x_2} u_2 \ar[r]_{\widetilde{f_1}} \ar[dd]|!{[dl];[drr]}\hole _(.3){\widetilde{\jmath_2}} & u_2 \ar[r]|!{[u];[dr]}\hole _(.3){f_2} \ar[dd]|!{[dll];[dr]}\hole _(.7){j_2} & x_3 
		\\
		[u_1(u_2u_3)]\times_{u_1} [u_1(u_2v_3)] \ar[dd]_{i_p}\ar[dr] \ar[rrr]^-{p} &
		&& (u_2u_3)\times_{x_2} (u_2v_3) \ar[dl] 
		\ar@/^1pc/!UR;[rr]_(.15){\ }="s2"^{f_3\widetilde{f_2}\pr_{12}} \ar@/_1pc/!DR;[rr]^(.15){\ }="t2"_{g_3\widetilde{f_2}\pr_{34}} 
		&& x_4
		\\
		&
		u_1 \ar[d]_{j_1} \ar[r] & x_2 
		\\
		[u_1(u_2u_3)]\times_{x_1} [u_1(u_2v_3)] \ar[r] &
		 x_1
		\ar@{=>}"s1";"t1"^{a_3}
		\ar@{=>}"s2";"t2"^{\lambda_{f_2}a_3}
	}}
\]
\caption{Constructing the double left whiskering}
\label{fig:double left whisker}
\end{figure}
Where $\lambda_{f_2}a_3$ is the unique 2-arrow satisfying 
\begin{equation}\label{eq:dbl_left_assoc_nat_1}
(\lambda_{f_2}a_3) 1_{i_q} = a_3 1_q.
\end{equation}
The last step is that $\lambda_{f_1}(\lambda_{f_2}a_3)$ is the unique 2-arrow satisfying 
\begin{equation}\label{eq:dbl_left_assoc_nat_2}
(\lambda_{f_1}(\lambda_{f_2}a_3)) 1_{i_p} = (\lambda_{f_2}a_3) 1_p.
\end{equation}
Note that from the diagram we also have
\begin{equation}\label{eq:dbl_left_assoc_nat_4}
pQ = i_qP.
\end{equation}

We next construct the 2-arrow $\lambda_{f_2\widetilde{f}_1}a_3$ specifying $\lambda_{(j_1\widetilde{\jmath_2},f_2\widetilde{f}_1)}a_3$, using the diagram in Figure~\ref{fig:one left whiskering}.
\begin{figure}
\[
\noindent\makebox[\textwidth]{%
\xymatrix{
	[u_1(u_2u_3)]\times_{x_1} [u_1(u_2v_3)] \ar[r]^-{\alpha} &[(u_1u_2)u_3]\times_{x_1} [(u_1u_2)v_3] 
	\ar@/^1pc/!UR;[rr]_(.2){\ }="s2"^{f_3(f_2\widetilde{f_1})^\sim\pr_1} \ar@/_1pc/!DR;[rr]^(.2){\ }="t2"_{g_3(f_2\widetilde{f_1})^\sim\pr_2} && x_4
	\\
	[u_1(u_2u_3)]\times_{u_1} [u_1(u_2v_3)] \ar[u]^{i_p}\\
	[u_1(u_2u_3)]\times_{u_1\times_{x_1}u_2} [u_1(u_2v_3)] \ar[r]_-{\widetilde{\alpha}} \ar[u]^{Q} & [(u_1u_2)u_3]\times_{u_1\times_{x_1}u_2} [(u_1u_2)v_3] \ar[r]_-{\pi} \ar[d] \ar[uu]^{i_\pi} & u_3v_3 \ar[d] \ar@/^1pc/[rr]_{\ }="s1"^{f_3\pr_1} \ar@/_1pc/[rr]^{\ }="t1"_{g_3\pr_2}&& x_4
	\\
	&u_1u_2 \ar[d] \ar[r]_{\widetilde{f_2}f_1} & x_3
	\\
	&x_1
	\ar@{=>}"s1";"t1"^{a_3}
	\ar@{=>}"s2";"t2"^{\lambda_{f_2\widetilde{f_1}}a_3}
}}
\]
\caption{Constructing the left whisking by the composite}
\label{fig:one left whiskering}
\end{figure}
Here $\lambda_{f_2\widetilde{f}_1}a_3$ is the unique 2-arrow satisfying 
\begin{equation}\label{eq:dbl_left_assoc_nat_3}
(\lambda_{f_2\widetilde{f}_1}a_3) 1_{i_\pi} = a_3 1_\pi,
\end{equation}
and we note the identities
\begin{align}
i_\pi\widetilde{\alpha} & = \alpha i_p Q \label{eq:dbl_left_assoc_nat_5}\\
qP & = \pi \widetilde{\alpha}\label{eq:dbl_left_assoc_nat_6}
\end{align}
Further, we have that the map of $\J$-fractions
\[
	\iota(a_{u_1,u_2,u_3})^{-1} + \lambda_{(j_1\widetilde{\jmath_2},f_2\widetilde{f}_1)}a_3 + \iota(a_{u_1,u_2,v_3}) 
\]
is specified by the 2-arrow $(\lambda_{f_2\widetilde{f}_1}a_3) 1_\alpha$.
Hence we can make the following calculation
\begin{align*}
(\lambda_{f_2\widetilde{f}_1}a_3) 1_\alpha 1_{i_pQ} &  = (\lambda_{f_2\widetilde{f}_1}a_3) 1_{i_\pi} 1_{\widetilde{\alpha}}\qquad \text{using \eqref{eq:dbl_left_assoc_nat_5} }\\
& = a_3 1_\pi 1_{\widetilde{\alpha}}\qquad \text{using \eqref{eq:dbl_left_assoc_nat_3}}\\
& = a_3 1_q 1_P \qquad \text{using \eqref{eq:dbl_left_assoc_nat_6} } \\
& = (\lambda_{f_2}a_3) 1_{i_q} 1_P \qquad \text{using \eqref{eq:dbl_left_assoc_nat_1} }\\
& = (\lambda_{f_2}a_3) 1_p 1_Q \qquad \text{using \eqref{eq:dbl_left_assoc_nat_4} }\\
& = (\lambda_{f_1}(\lambda_{f_2}a_3)) 1_{i_p} 1_Q \qquad \text{using \eqref{eq:dbl_left_assoc_nat_2} }
\end{align*}
Notice that $Q$ is co-ff as $\widetilde{\jmath_2}$ is a $\J$-cover, as well as the other two maps in the square defining $Q$, and we can apply Lemma~\ref{lemma:stronger_co-ff}. 
Since $i_p$ is also co-ff (by the proof of Lemma~\ref{lemma:left_whisker_combined}), and co-ff 1-arrows are closed under composition, we have $(\lambda_{f_1}(\lambda_{f_2}a_3)) = \lambda_{f_2\widetilde{f}_1}a_3 1_\alpha$, as we needed to show.

\subsection{Left+right whiskered}

We need to both whisker $a_2$ by $(j_1,f_1)$ and $(j_3,f_3)$, in that order, and also whisker it by $(j_3,f_3)$ and $(j_1,f_1)$, in that order.
The results then need to be vertically composed with associators in the appropriate order and compared.
Namely, we need to show the following 2-arrows are equal:
\[
	\lambda_{(j_1,f_1)}(\rho_{(j_3,f_3)}a_2) + \iota(a_{u_1,v_2,u_3})
	\stackrel{?}{=} 
	\iota(a_{u_1,u_2,u_3}) + \rho_{(j_3,f_3)}(\lambda_{(j_1,f_1)} a_3).
\]
What we will show, though, is the following:
\[
	 \rho_{(j_3,f_3)}(\lambda_{(j_1,f_1)} a_3)
	  \stackrel{?}{=} 
	 \iota(a_{u_1,u_2,u_3})^{-1} + \lambda_{(j_1,f_1)}(\rho_{(j_3,f_3)}a_2) + \iota(a_{u_1,v_2,u_3}),
\]
where $\iota(a_{u_1,u_2,u_3})^{-1} + \lambda_{(j_1,f_1)}(\rho_{(j_3,f_3)}a_2) + \iota(a_{u_1,v_2,u_3})$ is specified by the 2-arrow
\[
	\xymatrix{
		[(u_1u_2)u_3]\times_{x_1} [(u_1v_2)u_3] 
		\ar[r]^{\alpha} & [u_1(u_2u_3)]\times_{x_1} [v_1(u_2u_3)] \ar@/^1pc/!UR;[rr]_(.7){\ }="s" \ar@/_1pc/!DR;[rr]^(.7){\ }="t" && x_4 
		\ar@{=>}"s";"t"_{\lambda_{f_1}(\rho_{f_3}a_2)}
	}
\]
where $\lambda_{f_1}(\rho_{f_3}a_2)$ specifies the map $\lambda_{(j_1,f_1)}(\rho_{(j_3,f_3)}a_2)$ of $\J$-fractions.

We first construct the 2-arrow data of $\lambda_{(j_1,f_1)}(\rho_{(j_3,f_3)}a_2)$, using the diagram in Figure~\ref{fig:R+L whisker}.
\begin{figure}
\[
\xymatrix{
	&& u_1(u_2u_3) \ar[dr]_{\ }="s1"^{\widetilde{f}_2 \widetilde{f}_1^{\widetilde{f}_2}}\\
	u_1(u_2u_3) \times_{x_1} u_1(v_2u_3) \ar[urr] \ar[drr]&&& u_3 \ar@{=}[ddd] & \\
	&&u_1(v_2u_3)  \ar[ur]^{\ }="t1"_{\widetilde{g}_2 \widetilde{f}_1^{\widetilde{g}_2}}\\
	&& u_2u_3 \ar[dr]^{\widetilde{f}_2}_{\ }="s2"\\
	u_1(u_2u_3) \times_{u_1} u_1(v_2u_3) \ar[r]_-{p} \ar[uuu]^{i_p} \ar[dddd]& u_2u_3 \times_{x_2} v_2u_3 \ar[ur] \ar[dr]\ar[ddd]_{\widetilde{\jmath_3}^2} && u_3 \ar[r]_{f_3} \ar[ddd]^{j_3}& x_4\\
	&& v_2u_3 \ar[ur]_{\widetilde{g}_3}^{\ }="t2"\\
	&& u_2 \ar[dr]^{f_2}_{\ }="s3"\\
	& u_2\times_{x_2}v_2 \ar[ur] \ar[dr] \ar[d]&& x_3\\
	u_1 \ar[r]_{f_1} \ar[d] &x_2& v_2 \ar[ur]_{g_2}^{\ }="t3"\\
	 x_1
	\ar@{=>}"s1";"t1"_{\lambda_{f_1}(\widetilde{a}_2)}
	\ar@{=>}"s2";"t2"_{\widetilde{a}_2}
	\ar@{=>}"s3";"t3"_{a_2}
}
\]
\caption{Constructing the right then left whiskering}\label{fig:R+L whisker}
\end{figure}
We have $\rho_{f_3}a_2 := 1_{f_3}\widetilde{a}_2$ where $\widetilde{a}_2$ is the unique 2-arrow satisfying 
\begin{equation}\label{eq:LR_assoc_nat_2}
1_{j_3} \widetilde{a}_2 = a_2 1_{\widetilde{\jmath_3}^2}.
\end{equation}
Similarly $\lambda_{f_1}(\rho_{f_3}a_2) :=1_{f_3}\lambda_{f_1}(\widetilde{a}_2)$, where $\lambda_{f_1}(\widetilde{a}_2)$ is the unique 2-arrow satisfying 
\begin{equation}\label{eq:LR_assoc_nat_1}
\lambda_{f_1}(\widetilde{a}_2)1_{i_p} = \widetilde{a}_2 1_p.
\end{equation}

We next construct the 2-arrow data of $\rho_{(j_3,f_3)}(\lambda_{(j_1,f_1)} a_3)$, using the diagram in Figure~\ref{fig:L+R whisker}.
\begin{figure}
\[
\xymatrix{
	&& (u_1u_2)u_3 \ar[dr]_{\ }="s1"\\
	(u_1u_2)u_3\times_{x_1} (u_1v_2)u_3 \ar[urr] \ar[drr]  \ar[ddd]_{\widetilde{\jmath_3}^{12}} &&& u_3 \ar[r]^{f_3} \ar[ddd]^{j_3}& x_4\\ 
	&& (u_1v_2)u_3 \ar[ur]^{\ }="t1"\\
	&& u_1u_2 \ar[dr]_{\ }="s2"^{f_2\widetilde{f}_1} \\
	u_1u_2\times_{x_1} u_1v_2 \ar[urr] \ar[drr] &&& x_3\ar@{=}[ddd]\\
	&&u_1v_2 \ar[ur]^{\ }="t2"_{g_2\widetilde{f}_1}&&\\
	&& u_2 \ar[dr]^{f_2}_{\ }="s3"\\
	u_1u_2\times_{u_1}u_1v_2 \ar[r]_-{q}  \ar[uuu]^{i_q} \ar[d]& u_2v_2 \ar[ur] \ar[dr] \ar[d] && x_3\\
	u_1 \ar[r]_{f_1} \ar[d]_{j_1} & x_2 & v_2 \ar[ur]_{g_2}^{\ }="t3"\\
	x_1
	\ar@{=>}"s1";"t1"_{(\lambda_{f_1}a_2)^\sim}
	\ar@{=>}"s2";"t2"_{\lambda_{f_1}a_2}
	\ar@{=>}"s3";"t3"_{a_2}
}
\]
\caption{Constructing the left then right whiskering}
\label{fig:L+R whisker}
\end{figure}
We have $\rho_{f_3}(\lambda_{f_1}a_2) := 1_{f_3}(\lambda_{f_1}a_2)^\sim$, where $(\lambda_{f_1}a_2)^\sim$ is the unique 2-arrow satisfying
\begin{equation}\label{eq:LR_assoc_nat_4}
1_{j_3}(\lambda_{f_1}a_2)^\sim = (\lambda_{f_1}a_2)1_{\widetilde{\jmath_3}^{12}}
\end{equation}
Similarly, $\lambda_{f_1}a_2$ is the unique 2-arrow satisfying
\begin{equation}\label{eq:LR_assoc_nat_3}
(\lambda_{f_1}a_2)1_{i_q} = a_2 1_q.
\end{equation}
Finally, we note that we have commutative squares
\begin{equation}\label{eq:LR_assoc_nat_7}
\xymatrix{
	u_1u_2\times_{u_1}u_1v_2 \ar[d]_{i_q}&\ar[l]_-{J} (u_1u_2)u_3\times_{u_1} (u_1v_2)u_3 \ar[r]^{\widetilde{\alpha}}_{\simeq} \ar[d]_\pi &u_1(u_2u_3)\times_{u_1} u_1(v_2u_3) \ar[d]^{i_p} 
	\\
	u_1u_2\times_{x_1}u_1v_2 &\ar[l]^-{\widetilde{\jmath_3}^{12}} (u_1u_2)u_3\times_{x_1} (u_1v_2)u_3 \ar[r]_{\alpha}^{\simeq} &
	u_1(u_2u_3)\times_{x_1} u_1(v_2u_3)
}
\end{equation}
and we need to show that
\begin{equation}\label{eq:LR_whisker_final}
	\lambda_{f_1}(\rho_{f_3}a_2) 1_\alpha = 1_{f_3} (\lambda_{f_1}\widetilde{a}_2)1_\alpha = 1_{f_3}( \lambda_{f_1}\widetilde{a}_2)1_\alpha \stackrel{?}{=} 1_{f_3}(\lambda_{f_1}a_2)^\sim = \rho_{f_3}(\lambda_{f_1} a_2).
\end{equation}
We can now calculate:
\begin{align*}
1_{j_3}(\lambda_{f_1}\widetilde{a}_2)1_{i_p} & = 1_{j_3} \widetilde{a}_2 1_p \qquad \text{using \eqref{eq:LR_assoc_nat_1}}\\
& = a_2 1_{\widetilde{\jmath_3}^2} 1_p \qquad \text{using \eqref{eq:LR_assoc_nat_2}}\\
& = a_2 1_q 1_J 1_{\widetilde{\alpha}^{-1}}\\
& =(\lambda_{f_1}a_2)1_{i_q}1_J 1_{\widetilde{\alpha}^{-1}} \qquad \text{using \eqref{eq:LR_assoc_nat_3}} \\
& =(\lambda_{f_1}a_2)1_{\widetilde{\jmath_3}^{12}} 1_\pi 1_{\widetilde{\alpha}^{-1}}  \qquad \text{using\eqref{eq:LR_assoc_nat_7}}\\
& = 1_{j_3} (\lambda_{f_1}a_2)^\sim 1_\pi 1_{\widetilde{\alpha}^{-1}}\qquad \text{using \eqref{eq:LR_assoc_nat_4}}\\
& = 1_{j_3} (\lambda_{f_1}a_2)^\sim 1_{\alpha^{-1}} 1_{i_p} \qquad \text{using\eqref{eq:LR_assoc_nat_7} }.
\end{align*}
Since $j_3$ is ff and $i_p$ is co-ff, we have $(\lambda_{f_1}(\widetilde{a}_2))1_\alpha = (\lambda_{f_1}a_2)^\sim$, from which \eqref{eq:LR_whisker_final} follows.

\medskip

\noindent
This completes the proof of Lemma~\ref{lemma:associator_is_natural}.


\section{Proof that the middle-four interchange holds}\label{app:middle_four}

We need to show the following equality of vertical compositions in $K_\J$:
\[
	\diagram{
		\ar@/^1pc/[r]_{\ }="s1" \ar@/_1pc/[r]^{\ }="t1" & \ar[r]& 
		\ar@{=>}"s1";"t1"\\
		\ar[r]& \ar@/^1pc/[r]_{\ }="s2" \ar@/_1pc/[r]^{\ }="t2" &  
		\ar@{=>}"s2";"t2"
	}
	\equals
	\diagram{
		\ar[r]& \ar@/^1pc/[r]_{\ }="s2" \ar@/_1pc/[r]^{\ }="t2" &  
		\ar@{=>}"s2";"t2"\\
		\ar@/^1pc/[r]_{\ }="s1" \ar@/_1pc/[r]^{\ }="t1" & \ar[r]& 
		\ar@{=>}"s1";"t1"
	}
\]
(Here the 1-arrows are $\J$-fractions, and the 2-arrows are maps of $\J$-fractions.)
For the rest of this appendix, fix composable pairs of $\J$-fractions
\begin{align*}
	x_1 \xleftarrow{j_1} u_1 \xrightarrow{f_1} x_2\quad 
	x_2 \xleftarrow{j_2} u_2 \xrightarrow{f_2} x_3
	\\
	x_1 \xleftarrow{k_1} v_1 \xrightarrow{g_1} x_2\quad 
	x_2 \xleftarrow{k_2} v_2 \xrightarrow{g_2} x_3
\end{align*}
and maps of fractions $a_i\colon (j_i,f_i) \Rightarrow (k_i,g_i)$ for $i=1,2$. 
In everything that follows, unlabelled 1-arrows are canonical projection maps.

We first calculate  $\rho_{(j_2,f_2)}a_1 + \lambda_{(g_1,k_1)}a_2$.
Define the arrows $p\colon v_1u_2\times_{x_1} v_1v_2\to u_2v_2$, $i_p\colon v_1u_2\times_{v_1} v_1v_2 \to v_1u_2\times_{x_1} v_1v_2$, and $\widetilde{\jmath}_2 = \pr_{13}\colon$, and for $a,b\in \{u,v\}$, the notation $a_1b_2 := a_1\times_{x_2} b_2$.
The arrow $i_p$ is, by  Lemma~\ref{lemma:stronger_co-ff}, co-ff.
Using the definition of right whiskering (Definition~\ref{def:right whiskering}), and left whiskering as in Lemma~\ref{lemma:left_whisker_combined}, we get that the 2-cell representing $\rho_{(j_2,f_2)}a_1 + \lambda_{(g_1,k_1)}a_2$ is the unique descent along the co-ff arrow $u_1u_2\times_{x_1} (v_1u_2\times_{x_1} v_1v_2) \to u_1u_2\times_{x_1} v_1v_2$ (using Lemma~\ref{lemma:stronger_co-ff}) of
\begin{equation}\label{eq:middle_four_1}
	\diagram{
	& u_1u_2\times_{x_1} v_1u_2 \ar[dr]^(.9){\ }="t1" \ar[r]& u_1u_2 \ar[dr]_(.3){\ }="s1"^{\widetilde{f}_1^{j_2}} \\
	u_1u_2\times_{x_1} (v_1u_2\times_{x_1} v_1v_2) \ar[ur] \ar[dr] && v_1u_2 \ar[r]_{\ }="s2"^{\widetilde{g}_1^{j_2}} & u_2 \ar[r]^{f_2} & x_3\\
	&v_1u_2\times_{x_1} v_1v_2 \ar[r] \ar[ur] & v_1v_2 \ar[r]^{\ }="t2"_{\widetilde{g}_1^{k_2}} & v_2 \ar[ur]_{g_2}
	\ar@{=>}"s1";"t1"_{\widetilde{a}_1^{j_2}}
	\ar@{=>}"s2";"t2"_{\lambda_{g_1}a_2}
	}
\end{equation}
where $\widetilde{a}_1^{j_2}$ is the unique lift through $j_2$ of $a_1 1_{\widetilde{\jmath_2}}$ (using $j_2$ is ff), and $\lambda_{g_1}a_2$ is the unique descent of $a_1 1_p$ along the co-ff arrow $i_p$.

We next calculate $\lambda_{(j_1,f_1)}a_2 + \rho_{(k_2,g_2)}a_1$.
Define the arrows $q\colon u_1u_2\times_{u_1} u_1 v_2 \to u_2v_2$, $i_q\colon u_1u_2\times_{u_1} u_1 v_2\to u_1u_2\times_{x_1} u_1 v_2$, and $\widetilde{k}_2= \pr_{13}\colon u_1v_2\times_{x_1} v_1v_2 \to u_1v_1$.
The arrow $i_q$ is, by  Lemma~\ref{lemma:stronger_co-ff}, co-ff.
Then the 2-arrow representing $\lambda_{(j_1,f_1)}a_2 + \rho_{(k_2,g_2)}a_1$ is the unique descent along the co-ff arrow $u_1u_2\times_{x_1} (u_1v_2\times_{x_1} v_1v_2) \to u_1u_2\times_{x_1} v_1v_2$ (using Lemma~\ref{lemma:stronger_co-ff}) of
\begin{equation}\label{eq:middle_four_2}
\diagram{
	&u_1u_2\times_{x_1}u_1v_2 \ar[r] \ar[dr] &u_1u_2 \ar[r]_{\ }="s1"^{\widetilde{f}_1^{j_2}}  & u_2 \ar[dr]^{f_2}\\
	u_1u_2\times_{x_1} (u_1v_2\times_{x_1} v_1v_2) \ar[ur] \ar[dr] &&u_1v_2 \ar[r]^{\ }="t1"_{\widetilde{f}_1^{k_2}} &v_2 \ar[r]_{g_2} &x_3\\
	&u_1v_2\times_{x_1} v_1v_2 \ar[r] \ar[ur]_(.9){\ }="s2" & v_1v_2 \ar[ur]_{\widetilde{g}_1^{k_1}}^(.3){\ }="t2"
	\ar@{=>}"s1";"t1"_{\lambda_{f_1}a_2}
	\ar@{=>}"s2";"t2"_{\widetilde{a}_1^{k_2}}
}
\end{equation}
where $\widetilde{a}_1^{k_2}$ is the unique lift through $k_2$ of $a_1 1_{\widetilde{k_2}}$ (using $k_2$ is ff), and $\lambda_{f_1}a_2$ is the unique descent of $a_1 1_q$ along the co-ff arrow $i_q$.

From here, we will left whisker (in $K$) both of $\rho_{(j_2,f_2)}a_1 + \lambda_{(g_1,k_1)}a_2$ and $\lambda_{(j_1,f_1)}a_2 + \rho_{(k_2,g_2)}a_1$ by the co-ff 1-arrow
\[
	\pi\colon (u_1u_2)v_2\times_{x_1} (v_1v_2)u_2 \to u_1u_2\times_{x_1}v_1v_2
\]
(again using Lemma~\ref{lemma:stronger_co-ff}) and show the resulting 2-arrows are equal.
Here $(u_1u_2)v_2 = (u_1\times_{x_2})\times_{x_2} v_2$ and similarly for $(v_1v_2)u_2$.
From this it will follow that $\rho_{(j_2,f_2)}a_1 + \lambda_{(g_1,k_1)}a_2 =  \lambda_{(j_1,f_1)}a_2 + \rho_{(k_2,g_2)}a_1$, and hence that middle-four interchange holds.

Note first that using \eqref{eq:middle_four_1}, the 2-arrow $\left(\rho_{(j_2,f_2)}a_1 + \lambda_{(g_1,k_1)}a_2\right)1_\pi$ is equal to
\begin{align*}
&
\diagram@C-8pt{
	&& u_1u_2\times_{x_1} v_1u_2 \ar[dr]^(.9){\ }="t1" \ar[r]& u_1u_2 \ar[dr]_(.3){\ }="s1"^{\widetilde{f}_1^{j_2}} \\
	(u_1u_2)v_2\times_{x_1} (v_1v_2)u_2\ar[r] \ar[d] &u_1u_2\times_{x_1} (v_1u_2\times_{x_1} v_1v_2) \ar[ur] \ar[dr] && v_1u_2 \ar[r]_{\ }="s2"^{\widetilde{g}_1^{j_2}} & u_2 \ar[r]^{f_2} & x_3\\
	v_1u_2\times_{v_1} v_1v_2 \ar[rr]_{i_p} &&v_1u_2\times_{x_1} v_1v_2 \ar[r] \ar[ur] & v_1v_2 \ar[r]^{\ }="t2"_{\widetilde{g}_1^{k_2}} & v_2 \ar[ur]_{g_2}
	\ar@{=>}"s1";"t1"_{\widetilde{a}_1^{j_2}}
	\ar@{=>}"s2";"t2"_{\lambda_{g_1}a_2}
}\\
& \equals \diagram@C-8pt{
	& u_1u_2\times_{x_1} v_1u_2 \ar[r] \ar[dr]^(.9){\ }="t1" & u_1u_2 \ar[dr]_(.3){\ }="s1"^{\widetilde{f}_1^{j_2}} \\
	(u_1u_2)v_2\times_{x_1} (v_1v_2)u_2 \ar[ur] \ar[dr] && v_1u_2 \ar[r]_{\widetilde{g}_1^{j_2}} & u_2 \ar[r]^{f_2}_(.2){\ }="s2" & x_3\\
	&v_1u_2\times_{v_1} v_1v_2\ar[r]_-{p} & u_2v_2 \ar[ur] \ar[r] & v_2 \ar[ur]_{g_2}^(.3){\ }="t2"
	\ar@{=>}"s1";"t1"_{\widetilde{a}_1^{j_2}}
	\ar@{=>}"s2";"t2"_{a_2}
}\\
&\equals \diagram@C-8pt{
	(u_1u_2)v_2\times_{x_1} (v_1v_2)u_2 \ar@/^1pc/[]!R;[drr]_{\ }="s1" 
	\ar[dr]^(.8){\ }="t1" && 
	& u_2 \ar[r]^{f_2}_(.2){\ }="s2" & x_3\\
	&v_1u_2\times_{x_1} v_1v_2\ar[r]_-{p} & u_2v_2 \ar[ur] \ar[r] & v_2 \ar[ur]_{g_2}^(.3){\ }="t2"
	\ar@{=>}"s1";"t1"_{\widetilde{a}_1}
	\ar@{=>}"s2";"t2"_{a_2}
}
\end{align*}
where the last equality holds as $u_2v_2\to u_2$ is ff, so we can lift $\widetilde{a}_1^{j_2}$ to $\widetilde{a}_1$.
We note that $\widetilde{a}_1$ is a lift of $a_1 1_Q$ through $u_2v_2\to x_2$, for $Q\colon (u_1u_2)v_2\times_{x_1} (v_1v_2)u_2 \to u_1v_1$. 

Now using \eqref{eq:middle_four_2}, the 2-arrow $\left(\lambda_{(j_1,f_1)}a_2 + \rho_{(k_2,g_2)}a_1\right)1_\pi$ is equal to
\begin{align*}
&
\diagram@C-8pt{
	u_1u_2\times_{u_1} u_1 v_2\ar[rr]^{i_q}&&u_1u_2\times_{x_1}u_1v_2 \ar[r] \ar[dr] &u_1u_2 \ar[r]_{\ }="s1"^{\widetilde{f}_1^{j_2}}  & u_2 \ar[dr]^{f_2}\\
	(u_1u_2)v_2\times_{x_1} (v_1v_2)u_2\ar[r] \ar[u]&u_1u_2\times_{x_1} (u_1v_2\times_{x_1} v_1v_2) \ar[ur] \ar[dr] &&u_1v_2 \ar[r]^{\ }="t1"_{\widetilde{f}_1^{k_2}} &v_2 \ar[r]_{g_2} &x_3\\
	&&u_1v_2\times_{x_1} v_1v_2 \ar[r] \ar[ur]_(.9){\ }="s2" & v_1v_2 \ar[ur]_{\widetilde{g}_1^{k_1}}^(.3){\ }="t2"
	\ar@{=>}"s1";"t1"_{\lambda_{f_1}a_2}
	\ar@{=>}"s2";"t2"_{\widetilde{a}_1^{k_2}}
}\\
&\equals
\diagram@C-8pt{
	&u_1u_2\times_{u_1} u_1 v_2\ar[r]^-{q}&u_2v_2 \ar[r] \ar[dr]  & u_2 \ar[dr]^{f_2}_(.2){\ }="s1"\\
	(u_1u_2)v_2\times_{x_1} (v_1v_2)u_2\ar[ur] \ar[dr] &&u_1v_2 \ar[r]^{\widetilde{f}_1^{k_2}} &v_2 \ar[r]_{g_2}^(.3){\ }="t1" &x_3\\
	&u_1v_2\times_{x_1} v_1v_2 \ar[r] \ar[ur]_(.9){\ }="s2" & v_1v_2 \ar[ur]_{\widetilde{g}_1^{k_1}}^(.3){\ }="t2"
	\ar@{=>}"s1";"t1"_{a_2}
	\ar@{=>}"s2";"t2"_{\widetilde{a}_1^{k_2}}
}\\
&\equals
\diagram@C-8pt{
	&u_1u_2\times_{u_1} u_1 v_2\ar[r]^-{q}&u_2v_2 \ar[r] \ar[dr]  & u_2 \ar[dr]^{f_2}_(.2){\ }="s1"\\
	(u_1u_2)v_2\times_{x_1} (v_1v_2)u_2 \ar[ur]_(.8){\ }="s2" \ar@/_1pc/[]!R;[urr]^{\ }="t2"
	&& &v_2 \ar[r]_{g_2}^(.3){\ }="t1" &x_3
	\ar@{=>}"s1";"t1"_{a_2}
	\ar@{=>}"s2";"t2"_{\widetilde{a}'_1}
}
\end{align*}
where the last equality holds as $u_2v_2\to v_2$ is ff, so we can lift $\widetilde{a}_1^{k_2}$ to $\widetilde{a}'_1$.
But now $\widetilde{a}'_1$ is also a lift of $a_1 1_Q$ through $u_2v_2\to x_2$, hence $\widetilde{a}_1 = \widetilde{a}'_1$.
This give us $\left(\rho_{(j_2,f_2)}a_1 + \lambda_{(g_1,k_1)}a_2\right)1_\pi = \left(\lambda_{(j_1,f_1)}a_2 + \rho_{(k_2,g_2)}a_1\right)1_\pi$, and hence the desired result, using the fact $\pi$ is co-ff.

\section*{Acknowledgement}
The author was supported by the Australian Research Council's \emph{Discovery Projects} funding scheme (grant numbers DP130102578 and DP180100383), funded by the Australian Government, and also by {Mrs~R.} Thanks are due to an anonymous referee for suggesting improvements and noticing inadvertently ommitted details.

\bibliographystyle{amsplain}

\end{document}